\documentclass[a4paper,11pt]{article}

\usepackage[utf8]{inputenc} %
\usepackage[T1]{fontenc}    %
\usepackage{lmodern}
\usepackage{amsthm}
\usepackage{amsmath}
\usepackage{amssymb}
\usepackage{array}
\usepackage{dirtytalk}
\usepackage{url}
\usepackage{booktabs}
\usepackage{amsfonts}
\usepackage{mathtools}
\usepackage{nicefrac}
\usepackage{microtype}
\usepackage{subfiles}
\usepackage{bbold}         
\usepackage{xcolor}         
\usepackage{bm}            
\usepackage[font=footnotesize]{caption}
\usepackage{enumitem}
\usepackage[verbose=true,letterpaper]{geometry}
\usepackage{empheq}
\usepackage{framed}
\usepackage{soul}
\usepackage{xcolor}
\usepackage[normalem]{ulem}
\usepackage{tabularx}

\mathtoolsset{showonlyrefs}

\definecolor{myred}{HTML}{880000}
\definecolor{mygreen}{HTML}{008800}
\definecolor{myblue}{HTML}{000088}
\definecolor{linkblue}{HTML}{0000BB}

\usepackage{hyperref}
\hypersetup{
  colorlinks,
  linkcolor={myred},
  citecolor={myblue},
  urlcolor={linkblue}}

\newcommand{\R}{\mathbf R}

\newcommand{\E}{{\mathbf E}}
\renewcommand{\P}{\mathbf P}

\newcommand{\norm}[1]{\lVert {#1} \rVert}

\renewcommand{\leq}{\leqslant}
\renewcommand{\geq}{\geqslant}
\renewcommand{\le}{\leqslant}
\renewcommand{\ge}{\geqslant}

\newcommand{\argmin}{\mathop{\mathrm{arg}\,\mathrm{min}}}

\newcommand{\wh}{\widehat}

\newcommand{\di}{\mathrm{d}}
\DeclareMathOperator{\tr}{Tr}

\newcommand{\indic}[1]{\bm 1 ( #1 )}
\newcommand{\gaussdist}{\mathcal{N}}

\newcommand{\F}{\mathcal{F}}

\newcommand{\X}{\mathcal{X}}

\renewcommand{\top}{\mathsf{T}}

\usepackage{xspace}
\newcommand{\ie}{\textit{i.e.,}\@\xspace}
\newcommand{\eg}{e.g.,\@\xspace}
\newcommand{\iid}{i.i.d.\@\xspace}

\newcommand{\probas}{\mathcal{P}}

\newcommand{\localc}[2]{\mathsf{C}_{\mathsf{loc}} (#1, #2)}
\newcommand{\globalc}[2]{\mathsf{C}_{\mathsf{glob}} (#1, #2)}
\newcommand{\ew}{\mathsf{ew}}
\newcommand{\few}{\wh f^{\ew}}
\newcommand{\fq}{\wh f^{\mathsf{q}}}
\newcommand{\pew}{\wh \rho^{\ew}}
\newcommand{\pq}{\wh \rho^{\mathsf{q}}}
\newcommand{\fprog}{\wh f^{\mathsf{pm}}}

\newcommand{\opnorm}[1]{\| {#1} \|_{\mathrm{op}}}

\newcommand{\kl}{\mathrm{KL}}
\newcommand{\kll}[2]{\kl ({#1}, {#2})}

\newtheorem{proposition}{Proposition}
\newtheorem{theorem}{Theorem}
\newtheorem*{theorem*}{Theorem}
\newtheorem{lemma}{Lemma}

\theoremstyle{definition}

\newtheorem{example}{Example}

\theoremstyle{remark}

\newcommand{\innerp}[2]{\langle #1 , #2 \rangle}

\title{Local Risk Bounds for Statistical Aggregation}

\author{
  Jaouad Mourtada\thanks{CREST, ENSAE, Institut Polytechnique de Paris, France, \href{mailto:jaouad.mourtada@ensae.fr}{jaouad.mourtada@ensae.fr}}
  \qquad
  Tomas Va\v{s}kevi\v{c}ius\thanks{Institute of Mathematics, EPFL, Switzerland,
\href{mailto:tomas.vaskevicius@epfl.ch}{tomas.vaskevicius@epfl.ch}}
  \qquad
  Nikita Zhivotovskiy\thanks{Department of Statistics, University of California Berkeley, United States, \href{mailto:zhivotovskiy@berkeley.edu}{zhivotovskiy@berkeley.edu}}
}
\date{\today}

\begin{document}

\maketitle

\begin{abstract}

In the problem of aggregation, the aim is to combine a given class of base predictors to achieve predictions nearly as accurate as the best one.
In this flexible framework, no assumption is made on the structure of the class or the nature of the target.
Aggregation has been studied in both sequential and statistical contexts.
Despite some important differences between the two problems, the classical results in both cases feature the same \say{global} complexity measure.
In this paper, we revisit and tighten classical results in the theory of aggregation in the statistical setting by replacing the global complexity with a smaller, \say{local} one.
Some of our proofs build on the PAC-Bayes localization technique introduced by Catoni.
Among other results, we prove localized versions of the classical bound for the exponential weights estimator due to Leung and Barron [Trans. Inf. Theory., 52(8):3396–3410, 2006] and deviation-optimal  bounds for the $Q$-aggregation estimator. These bounds improve over the results of Dai, Rigollet and Zhang [Ann. Statist. 40(3):1878-1905, 2012] for fixed design regression and the results of Lecu\'{e} and Rigollet [Ann. Statist. 42(1):211-224, 2014] for random design regression.
\end{abstract}

\section{Introduction}
\label{sec:introduction}
Aggregation\let\thefootnote\relax\footnotetext{Accepted for presentation at the Conference on Learning Theory (COLT) 2023.} is a fundamental task in statistical learning theory and sequential prediction.
It can be informally described as follows: given a family
of predictors, combine their predictions to produce forecasts almost as accurate as those of the best predictor in the provided class.
This basic problem can be instantiated in a variety of situations, and provides a flexible building block to construct accurate and theoretically sound procedures.
The aggregation problem has been explored both in the statistical setting \cite{nemirovski2000nonparametric,yang2000mixing,catoni2004statistical,tsybakov2003optimal,leung2006information,dalalyan2008aggregation,audibert2008deviation,lecue2009aggregation,lecue2014q-aggregation,bellec2018optimal,mendelson2019unrestricted}
and in the sequential setting (see, \eg \cite{hannan1957approximation,foster1991prediction,vovk1998mixability, haussler1998sequential, azoury2001relative,vovk2001competitive} and the classic survey~\cite{cesabianchi2006plg}).
It is also related to the classical \emph{model selection} problem in statistics; see, \eg the lecture notes~\cite{massart2007concentration} for an introduction to the modern non-asymptotic treatment of the subject.

In this work, we focus on statistical variants of the aggregation problem, primarily within the frameworks of \emph{statistical learning/random-design regression} and \emph{denoising/fixed-design regression}.
We aim to relate and contrast these statistical problems with their sequential counterparts, emphasizing suitable complexity measures associated with each problem setting.
Before highlighting the differences between the two variants of aggregation, we start by discussing some important links and commonalities between them.

\paragraph{Sequential and statistical aggregation.}

For concreteness and to emphasize the complexity aspect of the problem, we will concentrate on regression with the squared loss.
Henceforth, let $\X$ be an input space and $\{ f_\theta : \theta \in \Theta \}$ be a class of functions $\X \to [-1, 1]$, called base predictors or experts.
The goal is to approximate the performance of the best function in this class.
An important special case occurs when $\Theta$ is a finite but otherwise arbitrary set.
In the statistical context, such a setting is called \emph{model aggregation} \cite{nemirovski2000nonparametric,tsybakov2003optimal}, also known as \emph{expert aggregation} in the sequential setting.

In the sequential setting, we are given a sequence of observations $(x_1, y_1), \dots, (x_n, y_n) \in \X \times [-1, 1]$ revealed one at a time. The aim is to construct a sequence of predictors $\wh f_0, \dots, \wh f_{n-1} : \X \to [-1, 1]$ such that $\wh f_i$ depends only on the first $i$ observations and for any data sequence, the \emph{regret}
\begin{equation}
  \label{eq:def-regret}
  \frac{1}{n} \sum_{i=1}^n \big( \wh f_{i-1} (x_i) - y_i \big)^2 - \min_{\theta \in \Theta} \frac{1}{n} \sum_{i=1}^n \big( f_{\theta} (x_i) - y_i \big)^2
\end{equation}
is upper bounded by an appropriate function of $M$ and $n$.

In contrast, in the statistical learning framework, the input-output pairs $(X_i, Y_i)_{1 \leq i \leq n}$ are assumed to be $n$ \iid random variables with the same distribution $P$ on $\X \times [-1, 1]$ as the generic pair $(X,Y)$.
The quality of a regression function $f : \X \to \R$ is then measured by its \emph{risk} $R (f) = \E [ (f (X) - Y)^2 ]$.
The distribution $P$ is unknown, but the sample $(X_i, Y_i)_{1 \leq i \leq n}$ is available.
The aim is to produce a predictor $\wh f_n : \X \to \R$ depending on the random sample that achieves a small
\emph{excess risk}
$$R (\wh f_n ) - \min_{\theta \in \Theta} R(f_\theta),$$
either in expectation or with high probability.

The fact that the sequential problem deals with arbitrary (deterministic) data suggests that guarantees in this setting can be transported to the statistical setting.
This is indeed partly the case, as shown by a standard conversion procedure (see \cite{barron1987bayes,littlestone1989online2batch,yang2000mixing,catoni2004statistical,cesabianchi2004onlinebatch}): for any sequence of predictors $\wh f_{0}, \dots, \wh f_{n}$ (with $\wh f_i$  based on the observations $(X_j, Y_j)_{1 \leq j \leq i}$) and any reference function $f : \X \to \R$, we have
\begin{equation}
  \label{eq:online-batch}
  \frac{1}{n+1} \sum_{i=1}^{n+1} \E [R (\wh f_{i-1})] - R (f)
  = \E \bigg[ \frac{1}{n+1} \sum_{i=1}^{n+1}  (\wh f_{i-1} (X_i) - Y_i)^2 - \frac{1}{n+1} \sum_{i=1}^{n+1}  (f (X_i) - Y_i)^2 \bigg],
\end{equation}
which holds because $\wh f_{i-1}$ is independent of $(X_i,Y_i)$.
Hence, the empirical regret with respect to $f_0$ controls the average excess risk of the predictors $\wh f_i$ for $1 \leq i \leq n+1$.
This conversion is, however, only partial for two reasons.
First, the obtained guarantee is only on the \emph{expected} risk rather than its deviations; as shown by Audibert~\cite{audibert2008deviation}, there may be a real difference between the expected and high-probability behavior of the risk.
Second, the excess risk bound obtained this way features the complexity measure of the harder, sequential version of the problem. This leads to suboptimal bounds, as discussed below.

\paragraph{Limitation of selection rules and the need for aggregation.}

An important aspect of aggregation theory is that it allows us to treat rather generic problems  without imposing strong structural assumptions. For instance, the class of base predictors 
may have unfavorable geometry due to being non-convex, and likewise,
the target to be predicted is not assumed to lie in some favorable position relative to the class of predictors.
An interesting consequence of this lack of structure
is that \emph{selection/proper learning} rules, namely, rules that select an element of the class of base predictors, cannot achieve optimal prediction accuracy.
In particular, this eliminates the most natural procedure -- that of empirical risk minimization -- as a candidate for optimal aggregation.

The necessity of improper learning is a fundamental facet of the theory of aggregation,  shared by both the sequential and statistical versions of the problem. More precisely, no proper/selection rule (deterministic or randomized) can improve upon the suboptimal worst-case regret or excess risk rate of order $\sqrt{\log M / n}$.
This lower bound holds due to the possible non-convexity of the class and the arbitrary position of the target $Y$ with respect to it.
Roughly speaking, non-convexity can induce an erratic behavior of selection rules and prevents one from exploiting the curvature of the squared loss.
In contrast, when the base class of predictors is convex, curvature of the loss function can be exploited to obtain a faster rate scaling as $\log M / n$ for empirical risk minimization.
For relevant lower bounds in the statistical setting, we refer
to \cite[\S3.5]{catoni2004statistical}, \cite[\S20.2]{anthony2009neural} and \cite{juditsky2008mirror,audibert2008deviation}.
A corresponding lower bound in the sequential setup, which can be deduced from the statistical lower bound by \eqref{eq:online-batch}, can be found in \cite[p.~72]{cesabianchi2006plg}.

\paragraph{Aggregation by exponential weights.}

Perhaps the most important aggregation procedure, both in the sequential and statistical settings, is the method of \emph{exponential weights}.
This estimator is useful and nontrivial even in the case of model aggregation, but it is most naturally defined and analyzed in a more general context described below.

Let $\{ f_\theta : \theta \in \Theta \}$ denote the class of base predictors, where the parameter set $\Theta$ is now general and not necessarily finite.
As before, we make no special assumptions on the structure of the class, which may be non-convex.
A natural idea to artificially introduce some structure to the class is to ``convexify'' it.
Specifically, let $\probas (\Theta)$ denote the set of probability measures on $\Theta$.
This set possesses a convex structure and $\Theta$ naturally embeds into it by associating each $\theta \in \Theta$ to the Dirac mass $\delta_\theta$. In addition, each function $h : \Theta \to \R$ defines a linear function on $\probas (\Theta)$ by integration: $\rho \mapsto \innerp{\rho}{h} = \int_{\Theta} h (\theta) \rho (\di \theta)$.
Now, let $\pi$ be a probability measure on $\Theta$, called ``prior'' in analogy with Bayesian statistics; this distribution encodes the ``inductive bias'' of the procedure, and should generally put mass on ``simpler'' regions of the class.
The prior $\pi$ endows $\probas (\Theta)$ with a convex regularizer given by the \emph{relative entropy} or \emph{Kullback-Leibler divergence} $\rho \mapsto \kll{\rho}{\pi} = \int_\Theta \log (\frac{\di \rho}{\di \pi}) \di \rho$, with the convention that $\kll{\rho}{\pi} = + \infty$ when $\rho$ is not absolutely continuous with respect to $\pi$.
The \emph{exponential weights} rule is obtained by optimizing the linearized objective on $\probas (\Theta)$ with an entropic regularization. Specifically,
the exponential weights posterior $\pew_n \in \probas (\Theta)$ with prior $\pi$ and inverse temperature (learning rate) $\beta > 0$ is given by
\begin{equation}
  \label{eq:def-exp-weights-variational}
  \pew_n
  = \argmin_{\rho \in \probas (\Theta)} \big\{ \innerp{\rho}{ R_n} + \beta^{-1} \cdot \kll{\rho}{\pi}  \big\}
\end{equation}
where $ R_n (\theta) = n^{-1} \sum_{i=1}^n (f_\theta (X_i) - Y_i)^2$ denotes the empirical risk of the predictor $f_{\theta}$.
The solution to this optimization problem can be expressed explicitly using the Donsker-Varadhan variational formula: %
\begin{equation}
  \label{eq:def-exp-weights-expression}
  \frac{\di \pew_n}{\di \pi} (\theta)
  = \frac{\exp \{ - \beta R_n (\theta) \}}{\int_{\Theta} \exp \{ - \beta  R_n (\vartheta) \} \pi (\di \vartheta)}
  .
\end{equation}
Finally, each distribution $\rho$ gives rise to the ``averaged'' or ``mixture'' predictor $f_\rho (x) = \int_\Theta f_\theta (x) \rho (\di \theta)$; in particular, the exponential weights predictor is $\few_n = f_{\pew_n}$.

In the model aggregation context with $|\Theta| = M$, the exponential weights algorithm with uniform prior on $\Theta$ and inverse temperature $\beta_i = i/8$ for $\few_{i}$ has a regret~\eqref{eq:def-regret} bounded by $8 \log (M) / n$ \cite[pp.~45-46]{cesabianchi2006plg}.
By identity~\eqref{eq:online-batch} and convexity of the risk, this implies that the following averaged estimator, proposed in~\cite{barron1987bayes,catoni2004statistical,yang2000mixing} and known as \emph{progressive mixture} rule:
\begin{equation}
  \label{eq:def-progressive-mixture}
  \fprog_n
  = \frac{1}{n+1} \sum_{i=0}^{n} \few_i
\end{equation}
achieves the excess risk bound
$\E [ R (\fprog_n) ] - \min_{\theta \in \Theta} R (f_\theta) \leq 8 \log (M) / (n+1)$, which is known to be 
the optimal rate %
of aggregation \cite{tsybakov2003optimal}.
Remarkably, while working in the space of measures (the convex hull of the class)
improves the geometry of the class by ensuring convexity, the above bound features the complexity of the original class rather than that of its convex hull.
Indeed, the best achievable excess risk over the convex hull of $M$ functions is substantially larger than the $\log M / n$ rate for model aggregation~\cite{tsybakov2003optimal}.

In fact, the progressive mixture was the only known rate-optimal model aggregation procedure until the works of~\cite{audibert2008deviation,lecue2009aggregation}, which introduced procedures with optimal high-probability bounds (that the progressive mixture does not achieve).%

\paragraph{Global and local entropic complexities.}

While the exponential weights/progressive mixture estimators achieve the optimal rates of aggregation for finite classes, entropic penalization is well-defined in a broader context. The class of ``entropy-regularized'' procedures leads to a natural ``entropic'' complexity measure, which is different from the more traditional measures of complexity such as covering numbers, Rademacher complexity, or VC dimension \cite{vapnik1999nature,koltchinskii2011oracle,gyorfi2002nonparametric,wainwright2019high}.

Entropic complexities depend on the prior $\pi$. For any $\pi$, it follows from the  regret bound for exponential weights in the sequential setting  together with the identity~\eqref{eq:online-batch}, that if $\beta_i = i/c$ with $c \geq 8$ then
\begin{equation}
  \label{eq:averaged-exponential-bound}
  \frac{1}{n+1} \sum_{i=0}^n \E [ R (\few_i) ]
  \leq \inf_{\gamma \in \probas (\Theta)} \bigg\{ \langle \gamma , R\rangle + c\cdot\frac{\kll{\gamma}{\pi}}{n+1} \bigg\}
  .
\end{equation}
The quantity on the right-hand side plays an important role in the study of exponential weights and entropy-regularized aggregation.
We call this quantity the \emph{global entropic complexity} \cite{vovk2001competitive,zhang2006entropy,audibert2009fastrates}. It is defined as
\begin{equation}
  \label{eq:global-kl-complexity}
  \globalc{\beta}{\pi}
  = \inf_{\gamma  \in \probas (\Theta)} \Big\{ \innerp{\gamma }{R} + \beta^{-1} \cdot \kll{\gamma }{\pi} \Big\}
  = - \frac{1}{\beta} \log \bigg( \int_\Theta e^{-\beta R (\theta)} \pi (\di \theta) \bigg)
  .
\end{equation}
It is worth noting that the complexity $\globalc{\beta}{\pi}$ only depends on $\beta$ and on the distribution of $R (\theta)$ under $\theta \sim \pi$,
and not explicitly on the ``size'' or dimension of the class.
In addition, the more mass the prior puts on parameters $\theta$ with small risk, the smaller the complexity; this quantity, therefore, measures the quality of the prior $\pi$ for the actual risk $R$.

The global entropic complexity appears in classical results in aggregation, both in the sequential and statistical settings (and within the statistical framework, both for learning and denoising).
For instance, in the statistical learning setting, the bound~\eqref{eq:averaged-exponential-bound} together with convexity of the risk $R$ implies that the progressive mixture~\eqref{eq:def-progressive-mixture} satisfies the following bound: $\E [R (\fprog_n)] \leq \globalc{(n+1)/c}{\pi}$.

At the same time, sequential prediction seems genuinely harder than statistical estimation, for the following reason.
In sequential prediction, one must issue predictions at each time step $i=1, \dots, n$, having seen only the first $i-1$ observations, and the error is averaged over sample sizes.
Clearly, less information is available at early rounds, which should result in a lower precision.
In contrast, in statistical estimation, one has at hand a full sample of $n$ observations.
That the global complexity $\globalc{\beta}{\pi}$ bounds the hardness of sequential prediction suggests that it provides an overly pessimistic account of the hardness of the easier, statistical problem.
This suboptimality manifests itself in different ways, such as additional logarithmic factors,
inability to handle unbounded parameter domains, or it can provide misleading indication on the qualitative effect of some problem parameters.

The refinement of the global entropic complexity that we consider is the following \emph{local entropic complexity}: letting $\pi_{h} = e^h \pi / \langle \pi, e^h\rangle \in \probas (\Theta)$ for a function $h : \Theta \to \R$, 
\begin{equation}
  \label{eq:def-local-complexity}
  \localc{\beta}{\pi}
  = \frac{\int_{\R^d} R (\theta) e^{- \beta R (\theta)} \pi (\di \theta)}{\int_{\R^d} e^{- \beta R (\theta)} \pi (\di \theta)}
  = \innerp{\pi_{-\beta R}}{R}
  .
\end{equation}
The local complexity
is the average risk under the exponential weights posterior with respect to the \emph{true} (rather than empirical) risk. In analogy with thermodynamics, the global complexity corresponds to the ``free energy'', while the local complexity corresponds to the ``average energy''.
The link between both complexities is given by the relations
\begin{equation}
  \label{eq:link-local-global}
  \frac{\di}{\di \beta} \{ \beta\cdot\globalc{\beta}{\pi} \} = \localc{\beta}{\pi}
  \quad \text{\ie} \quad
  \globalc{\beta}{\pi}
  = \frac{1}{\beta} \int_0^{\beta} \localc{\eta}{\pi} \di \eta
  .
\end{equation}
In addition, by a standard property of log-Laplace transforms, $- \beta \cdot \globalc{\beta}{\pi}$ is convex in $\beta$, so that $\localc{\beta}{\pi}$ is decreasing in $\beta$.
Combining this with the second identity in~\eqref{eq:link-local-global} shows that $\localc{\beta}{\pi} \leq \globalc{\beta}{\pi}$.
Another essentially equivalent quantity is
\begin{equation*}
  \inf_{\gamma \in \probas (\Theta)} \Big\{ \innerp{\gamma}{R} + \beta^{-1} \cdot \kll{\gamma}{\pi_{-\beta R}} \Big\}
  = \globalc{\beta}{\pi_{-\beta R}}
  = \frac{1}{\beta} \int_{\beta}^{2 \beta} \localc{\eta}{\pi} \di \eta,
\end{equation*}
(as %
since $\pi_{-\beta R})_{-\eta R} = \pi_{-(\beta+\eta)R}$.
In particular, by monotonicity of $\localc{\eta}{\pi}$ we have 
\begin{equation}
\label{eq:localvsglobalbound}
\localc{2\beta}{\pi} \le \globalc{\beta}{\pi_{-\beta R}} \le \localc{\beta}{\pi}.
\end{equation}
Combining inequality~\eqref{eq:averaged-exponential-bound} with the relation~\eqref{eq:link-local-global} leads to the bound
\begin{equation}
  \label{eq:ineq-averages}
  \frac{1}{n+1} \sum_{i=0}^{n} \E [ R (\wh f_i) ]
  \leq \frac{1}{n+1} \int_0^{n+1} \localc{c \,\eta}{\pi} \, \di \eta
  \leq
  \frac{1}{n+1} \sum_{i=0}^{n} \localc{c \, i}{\pi}.
\end{equation}
In words, the averaged (over sample sizes $i = 0, \dots, n$) %
risk of the exponential weights estimator $\wh f_i$ (based on the first $i$ observations and with inverse temperature $\beta_i = c\, i$) is bounded by the averaged local complexity.
Given this, it is natural to aim for the following strengthening of~\eqref{eq:ineq-averages}: control the risk of exponential weights (or another aggregation procedure) with $n$ samples by the local complexity $\localc{c \, n}{\pi}$.
However, this inequality does not follow from~\eqref{eq:ineq-averages}; instead, the classical approach is to deduce (using convexity of $R$) the global bound $\E [R(\fprog_n)] \leq \globalc{c (n+1)}{\pi}$ for the averaged estimator $\fprog_n$.
One of the main contributions of this work is to show that the risk of suitable entropy-regularized aggregation methods can indeed be controlled by the local complexity $\localc{c \, n}{\pi}$, at the same level of generality as the global bounds.

\paragraph{Related work.}

Local priors and entropic complexities were first introduced by Catoni~\cite{catoni2004statistical} in the context of PAC-Bayesian bounds for classification, with a similar approach proposed by Zhang~\cite{zhang2006information}.
Alquier~\cite{alquier2008pac} also obtained local PAC-Bayesian bounds for regression.
However, previous works \cite{catoni2004statistical,zhang2006information,alquier2008pac,alquier2021user} either produced generalization error bounds at a \say{slow rate} or assumed the target was in a favorable position relative to the class (Bernstein assumption) rather than in an arbitrary position as in the general aggregation problem.
This distinguishes our results from those of prior works.

Another relevant reference is~\cite{bellec2018optimal}, where the problem of aggregation is considered in the specific case of tuning the regularization parameter for ridge regression.
In this case, the authors obtain guarantees with an excess risk term that does not depend on the number of possible parameters in the discrete grid, or on the range of possible parameters.
These improvements are related to those obtained from local bounds.
However, the approach from~\cite{bellec2018optimal} seems to rely on the specific structure of the considered class, in particular on the fact that it is comprised of ordered (affine) spectral smoothers.

Guarantees featuring the relative entropy as a complexity measure are sometimes referred to as ``PAC-Bayesian'' in the literature~\cite{mcallester1999some,catoni2007pacbayes,alquier2021user}.
Specifically, PAC-Bayesian inequalities, first used by McAllester~\cite{mcallester1999some,mcallester2003pac}, are obtained by combining the Chernoff method for establishing deviation inequalities with duality for the relative entropy.
We note that this methodology has also been applied to study
the generalization error (difference between empirical and population risks) of learning algorithms, which is distinct from the
excess risk studied in this paper.
We refer to a sample of works in this direction \cite{mcallester1999some,mcallester2003pac,catoni2007pacbayes,dziugaite2017nonvacuous,steinke2020reasoning,lugosi2023online} and to the recent surveys
\cite{alquier2021user,raginsky202110} for further references.

The present work belongs to the literature on aggregation in the statistical setting, which has been considered for both fixed and random design.
Guarantees in expectation for the exponential weights estimator were obtained in~\cite{barron1987bayes,yang2000mixing,catoni2004statistical,juditsky2008mirror} (with additional averaging) for random design, and in~\cite{george1986minimax,leung2006information,dalalyan2008aggregation} for fixed design.
Sharp bounds in deviation for model aggregation were first obtained in~\cite{audibert2008deviation,lecue2009aggregation}; closer to our work, the method of $Q$-aggregation was shown in~\cite{rigollet2012kullback,dai2012q-aggregation,lecue2014q-aggregation,bellec2018optimal} to achieve the worst-case optimal rate of aggregation in deviation.
We also mention that additional aggregation methods have been proposed in~\cite{wintenberger2017boa,vanderhoeven2022},
using a sequential
procedure featuring a second-order correction related to that used in $Q$-aggregation; a related approach has been considered in the context of online learning in~\cite{audibert2009fastrates,koolen2015second,vanErven2021MetaGrad}.
Finally, extensions to heavy-tailed situations have been considered in \cite{mendelson2017aggregation, mendelson2019unrestricted, mourtada2021robust}.

\subsection{Summary of main results}
\label{sec:summary-of-main-results}

The rest of the paper is split into two parts. First, in Section~\ref{sec:main-results}, we state our localized PAC-Bayesian bounds in three different statistical learning settings. Then, in Section~\ref{sec:applications}, we instantiate our abstract results to concrete problems.
Our main contributions can be summarized as follows:
\begin{itemize}
\item In the fixed design regression setting, we prove two localized bounds. Theorem~\ref{thm:localized-exp-weights-gaussian-model} provides a localized version of the classical
  risk bound for the exponential weights estimator due to Leung and Barron \cite{leung2006information}.
  Theorem~\ref{thm:localized-fixed-design-q-aggregation} proves an analogous bound for the $Q$-aggregation estimator, this time optimal in deviation.
\item Our fixed design local risk bound for the $Q$-aggregation estimator is extended to the random design setting in Theorem~\ref{thm:pac-bayes-q-aggregation}.
  A special case of this bound resolves a question raised by Lecué and Mendelson \cite[Question 1.2]{lecue2013optimality} on model aggregation guarantees that adapt to the hardness of the problem instead of always paying for the worst-case minimax rate.
\item Section~\ref{sec:gauss-priors} 
  deals with Gaussian process priors, making explicit some distinctions between local and global complexities and connecting the $Q$-aggregation and ridge regression estimators.
\item Section~\ref{sec:improperridge} is devoted to the analysis of improper ridge-type estimators within the context of random design regression. Although we utilize a leave-one-out type of analysis in this section, we are able to recover a complexity measure similar to those that appear in Section~\ref{sec:gauss-priors}.
\item Theorem~\ref{thm:exchangablepriors} in Section~\ref{sec:binaryclass}  provides a deviation-optimal localized bound for the $Q$-aggregation estimator in the transductive learning setup with almost exchangeable priors. This result improves over the results from \cite{rakhlin2017empirical} for Vapnik-Chervonenkis classes by accounting for their local structure in the resulting bounds. 
\end{itemize}
The proofs and further technical details are deferred to the appendix. 

\subsection{Notation}
\label{sec:notation}

This section summarizes the notation used throughout the paper. We let $\Theta$ and $\X$ be measurable spaces, and $f : \Theta \times \X \to \R$ be a measurable function.
For any $\theta \in \Theta$, let $f_\theta = f (\theta, \cdot) : \X \to \R$.
We denote by $\probas (\Theta)$ the space of all probability distributions on $\Theta$.
Henceforth, we leave measurability assumptions implicit to lighten the presentation, and adopt the convention that all introduced functions between measurable spaces are measurable. In order to distinguish between expectations over $\Theta$ (which come from the fact that we work on the probability simplex $\probas (\Theta)$) and randomness that comes from random samples $(X,Y)$ in $\X \times \R$, we denote expectations over $\Theta$ using integral and measure-theoretic notation, while we use probabilistic notation for random variables such as random samples.
Specifically, for $\mu$ a measure on $\Theta$ and $g : \Theta \to \R$, we denote by
$
  \langle \mu, g\rangle
  = \int_{\Theta} g(\theta) \mu (\di \theta)
$
the integral of $f$ with respect to the measure $\mu$, assuming it is well-defined.
For any distribution $\rho \in \probas (\Theta)$, and any function $f : \Theta \times \X \to \R$, the notation $\langle \rho, f \rangle$ averages over $\Theta$ and results in a function $f_\rho : \X \to \R$ defined for any $x \in \X$ by
$f_\rho (x)
  = \langle \rho, f (\cdot, x) \rangle
  = \int_\Theta f_\theta (x) \rho (\di \theta).$ We introduce the shorthand notation $R(\rho) = R(f_{\rho})$ and $\langle \rho, R_n \rangle = \int_{\Theta} R(f_{\theta})\rho(\di\theta).$
The empirical norm $\|\cdot\|_{n}^{2}$ is defined by
$\|f\|_{n}^{2} = \frac{1}{n}\sum_{i=1}^{n}f(X_{i})^{2}$. For any function $g : \X \to \R$, and any probability measure $\rho \in \probas(\Theta)$, and any function $f : \Theta \times \X \to \R$ we also introduce the following shorthand notation 
$
  \langle \rho, \|f - g\|_{n}^{2} \rangle = \int_{\Theta} \|f(\theta, \cdot) - g(\cdot)\|_{n}^{2}\rho(\di\theta)
$. 
{For any functions $f, g$, we write $f \lesssim g$ if there is an absolute constant $c > 0$ such that $f \le cg$.}

\section{Main results}
\label{sec:main-results}

We now state the main results of this paper: localized PAC-Bayesian
bounds for various statistical aggregation problems. In
Section~\ref{sec:fixed-design}, we treat the fixed design (or
denoising) setting, where the covariates are assumed to be deterministic.
The case of random design regression (statistical learning)  is then considered in 
Section~\ref{sec:random-design}.
We also mention that a third setting, namely
transductive learning, will be considered in Section~\ref{sec:transductive}.

While each of the learning formulations listed above admits a different
population risk functional $R$, its empirical counterpart $R_{n}$ remains the
same across the three problems. 
Given a sample $(X_{i}, Y_{i})_{i=1}^{n}$ of $n$ observations, the empirical risk $R_{n}(f)$ of a function $f : \mathcal{X} \to \R$ is defined by
\begin{equation}
  \label{eq:empirical-risk}
  R_{n}(f) = \frac{1}{n}\sum_{i=1}^{n}(f(X_{i}) - Y_{i})^{2}.
\end{equation}

Our results apply to two entropy-regularized aggregation procedures.
The first one is the \emph{exponential weights} estimator. Given a prior
distribution $\pi \in \mathcal{P}(\Theta)$ and an inverse temperature parameter
$\beta > 0$, it is defined by
\begin{equation}
  \label{eq:exp-weights-estimator-definition}
  \few_{n} = \langle \pew_{n}, f \rangle, \text{ where }
  \pew_{n}
  =
  \pew_{n}(\pi, \beta)
  = \argmin_{\rho \in \mathcal{P}(\Theta)}\left\{
    \langle \rho, R_{n} \rangle
    + \frac{1}{\beta} \kl(\rho, \pi)
  \right\}.
\end{equation}
The estimator $\few_{n}$ is analyzed in Section~\ref{sec:fixed-design} in the
classical Gaussian fixed design setting. The purpose of this analysis is to
demonstrate some challenges associated with PAC-Bayesian
localization in the simplest possible setting and also, to provide an
improvement over the classical non-localized bound of Leung and Barron \cite{leung2006information}.

Despite optimal in-expectation performance in the denoising setting,
the exponential weights estimator $\few_{n}$ is known to suffer from
sub-optimal tail behavior \cite{dai2012q-aggregation}.
In the fixed design setting, a known remedy is to replace the linear term
$\langle \rho, R_{n} \rangle$ in \eqref{eq:exp-weights-estimator-definition}
by the convex combination
$\alpha \langle \rho, R_{n} \rangle + (1-\alpha) R_{n}(\rho)$ for any $\alpha \in (0,1)$.
The intuition for this manipulation is as follows. The choice $\alpha = 1$
yields the exponential weights estimator, which as mentioned above is deviation
sub-optimal. While minimizing the objective with the term $\langle \rho, R_{n}
\rangle$ introduces inductive bias towards ``sparse'' probability measures,
yielding good complexity regularization for the model selection aggregation
problem, this regularizing effect is completely lost at the other extreme
$\alpha = 0$.
On the other hand, any $\alpha \in (0,1)$ is a good choice: the resulting
objective still benefits from the inductive bias of the linear term $
\alpha \langle \rho, R_{n}\rangle$, while the term $(1-\alpha)R_{n}(\alpha)$ allows us
to benefit from the curvature of the quadratic loss that is lost with the
choice $\alpha = 1$; this curvature of the quadratic loss is crucial to
regularize variance terms that render the estimator $\few$ deviation suboptimal
despite its in-expectation optimality for the denoising problem.
While any $\alpha \in (0,1)$ would work, in order to simplify our results
we define the \emph{$Q$-aggregation} estimator with the particular choice $\alpha
= 1/2$:
\begin{equation}
  \label{eq:q-aggregation-estimator-definition}
  \fq_{n} = \langle \pq_{n}, f \rangle, \text{ where }
  \pq_{n}
  =
  \pq_{n}(\pi, \beta)
  = \argmin_{\rho \in \mathcal{P}(\Theta)}\left\{
    \frac{1}{2}\langle \rho, R_{n} \rangle
    +
    \frac{1}{2}R_{n}(\rho)
    + \frac{1}{\beta} \kl(\rho, \pi)
  \right\}.
\end{equation}
This estimator was first proposed by Audibert \cite{audibert2004aggregated} in the context of convex aggregation (where one aims to compare to convex combinations of functions in the class).
It was first recognized as an optimal model aggregation procedure by Rigollet~\cite{rigollet2012kullback},
and variants thereof were further analyzed in
\cite{dai2012q-aggregation, lecue2014q-aggregation, bellec2018optimal}.
In particular, Dai, Rigollet and Zhang \cite{dai2012q-aggregation} obtain a global PAC-Bayes bound in
the fixed design setup.

\subsection{Fixed design}
\label{sec:fixed-design}

In this section, let $X_{1},\dots,X_{n} \in \mathcal{X}$
be arbitrary \emph{deterministic} covariate vectors and consider the model:
\begin{equation}
  \label{eq:fixed-design-model}
  Y_{i} = f^{\star}(X_{i}) + Z_{i},
\end{equation}
where $Y = (Y_{1}, \dots, Y_{n}) \in \R^{d}$ is the vector of response
variables, $Z = (Z_{i}, \dots, Z_{n})$ is a random vector of zero-mean noise,
and $f^{\star}$ is the regression function.
Observe that the only randomness in the model \eqref{eq:fixed-design-model} comes from
the noise vector $Z$.
Define the population risk functional $R$ by
\begin{equation}
  \label{eq:fixed-design-R}
  R(f) = \mathbf{E}_{Y}\left[
    \frac{1}{n}\sum_{i=1}^{n} (f(X_{i}) - Y_{i})^{2}
  \right]
  = \|f - f^{\star}\|_{n}^{2} + \frac{1}{n}\sum_{i=1}^{n}\mathrm{Var}(Z_{i}),
\end{equation}
while its empirical counterpart $R_{n}$ is defined in
\eqref{eq:empirical-risk}.
With a slight abuse of notation, in what follows we identify $f^{\star}$ with
the vector $(f^\star (X_i))_{1 \leq i \leq n} \in \R^n$.

This section contains two results. The following result, proved in
Section~\ref{sec:proof-fixed-design-exp-weights}, provides the first localized
in-expectation bound for the exponential weights estimator under the assumption
that the noise vector $Z$ is standard normal. In particular, this yields an
improvement over the classical \emph{global} PAC-Bayesian bound obtained in the same
setup by Leung and Barron~\cite{leung2006information}.
The assumption that $Z$ is Gaussian, in particular, allows us to leverage Stein's Unbiased Risk Estimator (SURE) in the analysis of $\few$. Such an approach in the context of aggregation via
exponential weights was first considered by George~\cite{george1986minimax}, and
later systematically applied by Leung and Barron~\cite{leung2006information}, Dalalyan and Tsybakov~\cite{dalalyan2008aggregation} and Dalalyan and Salmon~\cite{dalalyan2012sharp}, among others.
For a treatment of more general noise distributions, we refer to~\cite{dalalyan2022simple} and references therein.

\begin{theorem}
  \label{thm:localized-exp-weights-gaussian-model}
  Consider the fixed design setting described above
  and assume that $(Z_{i})_{i=1}^{n}$ are \iid centered
  Gaussian random variables with variance $\sigma^{2}$.
  Let $\{ f_\theta : \theta \in \Theta\}$ be a family of vectors in $\R^n$ and $\pi \in \probas (\Theta)$ a fixed probability distribution.
  Then, the exponential weights estimator $\few_{n} = \few_{n}(\pi, \beta)$ defined in~\eqref{eq:exp-weights-estimator-definition}, with prior $\pi$ and inverse temperature $\beta = n/(8\sigma^{2})$, satisfies
  \begin{equation}
    \label{eq:bound-exp-weights-fixed}
    \E \Big[
      \|\few_{n} - f^{\star}\|_{n}^{2}
    \Big]
    \leq
    \int_{\Theta} \norm{f_\theta - f^\star}_n^2 \, \pi_{- n R / (16 \sigma^2)} (\di \theta)
    .
  \end{equation}
\end{theorem}

Equivalently and with the notation~\eqref{eq:def-local-complexity} from the introduction, the bound~\eqref{eq:bound-exp-weights-fixed} writes
  \begin{equation*}
    \E \big[ R (\few_n) \big]
    \leq \localc{n/(16 \sigma^2)}{\pi}.
  \end{equation*}

As already mentioned at the beginning of Section~\ref{sec:main-results}, the
exponential weights estimator is deviation suboptimal. On the other hand,
it was shown in~\cite{dai2012q-aggregation} showed that the $Q$-aggregation estimator satisfies
a \emph{global} exponential-tail bound in the fixed design setting with
sub-Gaussian noise; we say that a centered real random variable $Z$ is
$\sigma^{2}$-sub-Gaussian if for any $\lambda \in \R$ we have
$\E \exp(\lambda Z) \leq \exp(\lambda^{2}\sigma^{2}/2)$.
The following result shows that in fact, the $Q$-aggregation
estimator satisfied a deviation optimal \emph{local} bound. See
Appendix~\ref{sec:proof-q-aggregation-fixed-design} for the proof.

\begin{theorem}
  \label{thm:localized-fixed-design-q-aggregation}
  Consider the fixed design setting
  and assume that the noise random variables $Z_{i}$ are \iid, zero-mean and
  $\sigma^{2}$-sub-Gaussian.
  Let $\{ f_\theta : \theta \in \Theta\}$ be a family of vectors in $\R^n$ and $\pi \in \probas (\Theta)$ a fixed probability distribution.
  Let $\fq_{n} = \fq_{n}(\pi, \beta)$ denote the $Q$-aggregation estimator defined in~\eqref{eq:q-aggregation-estimator-definition}, with prior $\pi$ and inverse temperature $\beta = n/(12\sigma^{2})$.
  Then, for any $\delta \in (0,1)$, with probability at least $1-\delta$
  it holds that
  \begin{align}
    \|\fq_{n} - f^{\star}\|_{n}^{2}
    &\leq
    \int_{\Theta} \norm{f_\theta - f^\star}_n^2 \, \pi_{- n R/(36 \sigma^2)} (\di \theta)
      + \frac{18\sigma^{2} \log(1/\delta)}{n}
    .
  \end{align}
\end{theorem}

\subsection{Random design}
\label{sec:random-design}

We now to to the \emph{random design} setting, where we observe $n$ input-output pairs $(X_{i},
Y_{i})$, $i=1,\dots,n$, each drawn independently from some distribution $P$.
Unlike in the fixed design setting considered in Section~\ref{sec:fixed-design}, we work under boundedness assumptions on both the response variable $Y$ and the functions in the class.
Specifically, we assume in this section that, for some constant $b > 0$, one has $|Y| \leq b$ almost surely, and for every $\theta \in \Theta$,
\begin{equation}
  \label{eq:boundedassumption}
  \norm{f_\theta}_\infty
  = \sup_{x \in \mathcal{X}} |f_{\theta} (x)|
  \leq b.
\end{equation}
When working with arbitrary classes of functions, such boundedness conditions on the response variable $Y$ and the reference functions are natural and have been considered by many authors; see, e.g., \cite{tsybakov2003optimal, audibert2008deviation, lecue2014q-aggregation}.
Finally, the population risk functional $R$ is defined by
$
  R(f) = \E_{(X,Y) \sim P}\left[
    (f(X) - Y)^2
  \right],
$
while its empirical version $R_{n}$ is defined in \eqref{eq:empirical-risk}.
The following result in proved in Appendix~\ref{sec:proof-random-design-localization}.

\begin{theorem}
  \label{thm:pac-bayes-q-aggregation}
  Consider the random design setup described above.
  Let $\pi \in \probas (\Theta)$ be any fixed probability distribution
  and let $\beta = c_1 n/ b^{2}$ be the inverse temperature parameter.
  Let $\fq_{n} = \fq_{n}(\pi, \beta)$ denote the $Q$-aggregation estimator defined in
  \eqref{eq:q-aggregation-estimator-definition}.
  Then, for any $\delta \in (0,1)$, with probability at least $1-\delta$
  it holds that
  \begin{equation}
      R(\fq_{n})
      \leq
      \int_{\Theta} R (f_\theta) \, \pi_{- c_1 n R/ (3 b^2)} (\di \theta)
      + \frac{c_2 b^2 \log(3/\delta)}{n},
  \end{equation}
  where $c_1, c_2 > 0$ are universal constants.
\end{theorem}

Finally, it is important to note that the estimator $\fq_{n}$, within the framework of model aggregation involving random design, has not been previously analyzed. The question of analyzing the excess risk associated with this particular estimator was posed in \cite{mehta2017expconcave}.

\section{Applications}
\label{sec:applications}

In this section, we instantiate the local bounds in several different contexts.

\subsection{Model aggregation}
\label{sec:finite-aggregation}

As discussed in the introduction, a problem where $\Theta = \{\theta_{1}, \dots, \theta_{M}\}$ is a finite set is called \emph{model aggregation}. While the results we are about to discuss are applicable in all three learning setups considered in Section~\ref{sec:main-results}, for the sake of concreteness, let us consider the random design setup described in Section~\ref{sec:random-design}. In particular, the response variable $Y$ is almost surely contained in $[-b, b]$ and the reference class of functions satisfies the boundedness condition \eqref{eq:boundedassumption}. %

Let $\widehat{f} = \widehat{f}(S_{n})$ denote the output of any statistical estimator, where $S_{n} = (X_{i}, Y_{i})_{i=1}^{n}$ is the observed data sample. In the model aggregation problem, we wish to obtain  bounds on 
\[
  R(\widehat{f}) - \min_{\theta \in \Theta} R(f_{\theta}).
\]
The following minimax lower bound is well-known \cite{tsybakov2003optimal}:
\begin{equation}
    \label{eq:aggregation-minimax-lowerbound}
    \inf_{\widehat{f}} \sup_{P, |\Theta|=M} \E R(\widehat{f}) - \min_{\theta \in \Theta} R(f_{\theta})
     \geq c\frac{b^2 \log M}{n},
\end{equation}
  where $c >0$ is a universal constant, the expectation is taken with respect
  to $S_{n}$, the infimum is taken over all statistical estimators
  $\widehat{f}$, and the supremum is taken over all data-generating
  distributions and reference classes of functions of size $M$ that respect the
  above-described constraints.

Lecué and Mendelson~\cite{lecue2013optimality} asked whether it is possible to prove an upper bound for the model aggregation problem that exhibits some form of ``instance adaptivity'', that is, adaptivity to the hardness of the specific problem at hand.
This adaptivity should be quantitatively reflected in a risk bound that does not depend on the full size of the reference set $\Theta$, but only on the number of functions that are close to the best function in the class.
The proposition stated below provides such a guarantee.
We remark that partial results were obtained in~\cite{lecue2013optimality} under the so-called Bernstein assumption.

\begin{proposition}
  \label{prop:model-aggregation-random-design}
  Consider the random design model aggregation setup described above and let $\theta^{\star} \in \argmin_{j=1,\dots,M} R(\theta_{j})$.
  Let $\pi \in \probas (\Theta)$ be the uniform distribution over the finite set $\Theta$
  and set the inverse temperature parameter $\beta = n/(576 b^{2})$.
  Let $\fq_{n} = \fq_{n}(\pi, \beta)$ denote $Q$-aggregation estimator defined in
  \eqref{eq:q-aggregation-estimator-definition}.
  Then, for any $\delta \in (0,1)$, with probability at least $1-\delta$
  it holds that
  \begin{equation}
      R(\fq_{n}) - R(f_{\theta^{\star}})
      \leq c_{1} \frac{b^{2}}{n}\bigg[
        \log\bigg(\sum_{j=1}^{M}\exp\left(
            -c_{2}\frac{n}{b^2} \big(R(\theta_{j}) - R(\theta^{\star}) \big)\right)
           \bigg)
          +
           \log(1/\delta)
    \bigg],
  \end{equation}
  where $c_{1},c_{2} > 0$ are universal constants.
\end{proposition}

\begin{proof}
  The result follows from Theorem~\ref{thm:pac-bayes-q-aggregation} by upper bounding the local complexity using the inequality \eqref{eq:localvsglobalbound} and by plugging in $\gamma = \delta_{\theta^{\star}}$, a Dirac mass at $\theta^{\star}$, into the resulting bound.
\end{proof}

\subsection{Gaussian priors}
\label{sec:gauss-priors}

In this section, we consider the case where the prior $\pi$ on functions $f : \X \to \R$ is a centered Gaussian process.
In order to simplify notation, we assume that $\pi$ is supported on a finite-dimensional space of functions; this is not critical to our analysis, which would apply to the infinite-dimensional case with minor changes in terminology.
Furthermore, up to a change the of variable $x$, we may assume that $\X = \R^d$ and $\F$ is the class of linear functions $f_\theta (x) = \langle \theta, x\rangle$ for $\theta \in \R^d$; up to a linear transformation, we may assume that $\pi = \gaussdist (0, \gamma^{-1} I_d)$ for some $\gamma > 0$.

We first introduce some notation for both fixed and random design.
The sample covariance matrix is $\wh \Sigma_n = n^{-1} \sum_{i=1}^n X_i X_i^\top$, while for $\lambda > 0$ the \emph{ridge estimator} $\wh \theta_\lambda$ is:
\begin{equation*}
  \wh \theta_\lambda
  = \argmin_{\theta \in \R^d} \big\{ R_n (f_\theta) + \lambda \norm{\theta}^2 \big\}
  = (\wh \Sigma_n + \lambda I_d)^{-1} \cdot \frac{1}{n} \sum_{i=1}^n Y_i X_i
  .
\end{equation*}
We now determine the $Q$-aggregation estimator under Gaussian prior for the linear class.
\begin{proposition}
  \label{prop:q-aggregation-gaussian}
  Under the prior $\pi = \gaussdist (0, \gamma^{-1} I_d)$
  and with inverse temperature $\beta / 2$, the $Q$-aggregation posterior is $\gaussdist (\wh \theta_\lambda, \beta^{-1} (\frac{1}{2}\wh \Sigma_n + \lambda I_d )^{-1})$ while the exponential weights posterior is $\gaussdist (\wh \theta_\lambda, \beta^{-1} (\wh \Sigma_n + \lambda I_d )^{-1})$, where $\lambda = \gamma/\beta$.
  The associated prediction functions both coincide with that of the ridge regression estimator $f_{\wh \theta_\lambda}$.
\end{proposition}

Exponential weights is well-known in this context, and can be computed directly thanks to its explicit expression.
In contrast, the solution to the $Q$-aggregation optimization problem does not have a closed form in general.
In this case, the solution is obtained by first showing that it is Gaussian (using Proposition~\ref{prop:representer}), and then solving the optimization problem among Gaussian posteriors.

Next, we express and compare the global and local complexities.

\begin{proposition}
\label{prop:local-global-gaussian}
Assume that $\X = \R^d$ and let
\[
\F = \{x \mapsto \langle \theta, x\rangle: \theta \in \R^d\}
\]
be the class of linear functions. Set $\theta_\lambda = \argmin_{\theta \in \R^d} \{ R (\theta) + \lambda \norm{\theta}^2 \}$, and $\Sigma = \E [X X^\top]$ (assumed finite) for random design, and $\Sigma = \wh \Sigma_n = n^{-1} \sum_{i=1}^n X_i X_i^\top$ for fixed design. Set $\pi = \gaussdist (0, \gamma^{-1} I_d)$. Then, for any $\beta > 0$, letting $\lambda = \gamma/\beta$,
  \begin{align}
    \label{eq:global-kl}
    \globalc{\beta/2}{\pi}
    &=
      R (\theta_\lambda) + \lambda \norm{\theta_\lambda}^2 + \frac{1}{\beta} \log \det \big( I_d + \lambda^{-1} \Sigma \big)
    \\
    \localc{\beta/2}{\pi}
    &=
      R (\theta_\lambda) + \frac{1}{\beta} \cdot \tr [ (\Sigma + \lambda I_d)^{-1} \Sigma]
      .
  \end{align}
  In addition, %
  assuming for the second inequality below that $0 < \lambda \leq \opnorm{\Sigma}$,
  \begin{equation}
    \label{eq:trace-log-det}
    \tr \big[ (\Sigma + \lambda I_d)^{-1} \Sigma \big]
    \leq \log \det \big( I_d + \lambda^{-1} \Sigma \big)
    \leq 2 \log \bigg( 1 + \frac{\opnorm{\Sigma}}{\lambda} \bigg) \cdot \tr \big[ (\Sigma + \lambda I_d)^{-1} \Sigma \big]
    .
  \end{equation}
\end{proposition}

We note that the global entropic complexity already appears in the seminal work of~\cite{vovk2001competitive} in online linear regression, where it bounds the regret of the so-called Vovk-Azoury-Warmuth forecaster.
Here, the local complexity improves the global complexity in two ways.
First, it replaces the log determinant by a trace, allowing in particular to take $\lambda \to 0$; this allows to remove a $\log n$ factor as $n \to \infty$, and also to obtain bounds for arbitrary large norms of the comparison parameter $\theta$ (by taking $\lambda$ small enough that $\lambda \norm{\theta}^2 \ll d/n$).
Second, it removes an approximation term of $\lambda \norm{\theta_\lambda}^2$, which allows to obtain better rates in some regimes (see, \eg~\cite{blanchard2018optlip}). We also note that in the context of robust estimation, versions of local entropic complexities
have been computed in \cite{audibert2011robust, audibert2010linear}. 

Theorems~\ref{thm:localized-exp-weights-gaussian-model} and~\ref{thm:localized-fixed-design-q-aggregation} in the fixed design setting apply to the present context, and provide bounds for the ridge regression estimator featuring the local complexity of Proposition~\ref{prop:local-global-gaussian}. 
However, the random-design result does not apply to the Gaussian setting, since it requires boundedness of the base functions.
In Theorem \ref{thm:improper-ridge}, we indicate a different approach giving bounds for random design.

\subsection{Random-design bounds for ridge-type improper estimators}
\label{sec:improperridge}

In this section, we consider an alternative approach to obtaining excess risk bounds featuring related complexities as the one from the previous section, in the random design setting.

At this point, it is worth discussing some difficulties that are specific to the random design case.
Unlike for fixed design, we now aim to predict for a new, unseen point $X$.
This difference has the technical implication that, when analyzing an estimator like ridge regression, which is based on empirical risk minimization and corresponds to $Q$-aggregation as per Proposition~\ref{prop:q-aggregation-gaussian}, we must relate two distinct covariance matrices: the sample covariance matrix $\wh \Sigma_n$ (which measures the closeness of predictions of different parameters on the sample) and the population covariance $\Sigma$ (which controls closeness on a new random sample $X$).
To obtain guarantees of this nature, it is not hard to show that some assumptions on the distribution of $X$ are required beyond just the existence of $\Sigma$, such as controlled moments of order $p>2$.

Here, a key point is that precise excess risk bounds can nonetheless be obtained without additional assumptions on the distribution of $X$, provided that different estimators than ridge regression are considered; we refer to~\cite{gyorfi2002nonparametric,forster2002relative,tsybakov2003optimal,mourtada2021robust} for more information on this topic, in the case of linear regression (that is, with guarantees uniform over the linear class).

An important aspect of the problem at hand is that, in order to obtain nontrivial guarantees without additional assumptions on the distribution of $X$, one must consider aggregation procedures that return a function outside the class, that is, a nonlinear function~\cite{shamir2015sample,mourtada2021robust}.
While this is similar to the need for aggregation rules for finite classes discussed in the introduction, it calls for some explanation.
Indeed, for finite classes, the need for aggregation rules arises due to nonconvexity of the class; in contrast, the class of linear functions is convex.
In this case, the need for aggregation methods 
 stems from the unbounded nature of linear functions, leading to large errors when $X$ distribution can take large values with small probabilities~\cite{mourtada2021robust}. Hence, the guarantees of Section~\ref{sec:random-design} do not apply as Theorem~\ref{thm:localized-fixed-design-q-aggregation} needs bounded functions in the class. To ensure boundedness, we can truncate the original functions~\cite{gyorfi2002nonparametric,mourtada2021robust}, projecting their predictions to a bounded interval. However, this operation makes the resultant truncated linear functions class nonconvex, even if the original class was convex.

A natural approach to obtain guarantees is then to perform $Q$-aggregation on the class of \emph{truncated} linear functions, under a Gaussian prior on parameters.
One could then apply Theorem~\ref{thm:pac-bayes-q-aggregation}, since its boundedness condition does apply.
However, this approach has two main limitations.
First, the resulting estimator 
involves a non-convex minimization problem, whose computational complexity could depend exponentially on the dimension.
Second, it is unclear that the resulting procedure would satisfy a risk bound featuring the local complexity from Proposition~\ref{prop:local-global-gaussian}. 
Indeed, if one applies Theorem~\ref{thm:localized-fixed-design-q-aggregation}, the resulting bound would depend on the local entropic complexity (under Gaussian prior) of \emph{truncated} linear functions.
Assuming that $|Y| \leq b$, truncating linear functions at $b$ only reduces their risk, and thus also the associated global entropic complexity.
Hence, a bound based on the global complexity of Proposition~\ref{prop:local-global-gaussian} holds for this method.
In contrast, replacing functions by functions with smaller risk does not necessarily decrease the \emph{local} complexity, and as such the local entropic complexity of truncated linear functions could be larger than that of linear functions.

We now describe an alternative approach, which partly circumvents these limitations.
This approach is based on leave-one-out and exchangeability considerations, and has been used to design aggregation procedures for linear regression in~\cite{forster2002relative}, and for density estimation and logistic regression in~\cite{mourtada2022logistic}.
This approach allows to obtain bounds of ``local'' nature, without the additional logarithmic factors associated with ``global'' bounds.
In addition, it generally leads to estimators that are simple efficiently computable.
On the other hand, it only leads to guarantees in expectation rather than in deviation~\cite{mourtada2021robust}.
In addition, the resulting bound is slightly larger than the local entropic complexity of Proposition~\ref{prop:local-global-gaussian}.

We now describe the considered procedures and their guarantees.
Recall that the \emph{ridge} estimator is given by
\begin{equation}
  \label{eq:def-ridge}
  \wh \theta_\lambda
  = \argmin_{\theta \in \R^d} \big\{  R_n (\theta) + \lambda \norm{\theta}^2 \big\}
  = (\wh \Sigma_n + \lambda I_d)^{-1} \cdot \frac{1}{n} \sum_{i=1}^n Y_i X_i
  .
\end{equation}
For any $\lambda > 0$ and $x \in \R^d$, define the \emph{ridge leverage score} of $x \in \R^d$ by
\begin{equation}
  \label{eq:ridge-leverage}
  h_\lambda (x)
  = \bigg\langle \bigg( \sum_{i=1}^n X_i X_i^\top + \lambda n I_d + x x^\top \bigg)^{-1} x, x \bigg\rangle,
\end{equation}
and define (letting $\lambda' = (1+1/n)\lambda$)
\begin{equation}
  \label{eq:ridge-fw-smp}
  \wh f_\lambda (x)
  = (1 - h_{\lambda'} (x))^2 \cdot\langle \wh \theta_{\lambda'}, x\rangle
  .
\end{equation}
This estimator is a ridge-regularized analogue of the Forster-Warmuth procedure introduced in~\cite{forster2002relative}, which corresponds to the case $\lambda = 0$.
We therefore call it \emph{ridge FW}.
It can also be seen a linear version of the ridge-regularized logistic regression procedure from~\cite{mourtada2021robust}.
Also, for $b > 0$
define the truncation function $\psi_b : \R \to \R$ by $\psi_b (x) = \max (-b, \min (b, x))$ for $x \in \R$.
Then,
let
\begin{equation}
  \label{eq:ridge-truncated}
  \wh g_{\lambda, b} (x)
  = \psi_b \big( \innerp{\wh \theta_{\lambda'}}{x} \big)
  .
\end{equation}
This procedure is a ridge-regularized version of the truncated least squares estimator considered in~\cite{gyorfi2002nonparametric,mourtada2021robust}.

These estimators satisfy the following excess risk bounds:
\begin{theorem}
  \label{thm:improper-ridge}
  For any distribution $P$ on $\R^d \times \R$ such that $|Y| \leq b$
  almost surely and $\Sigma = \E X X^\top$ exists,
  \begin{align}
    \label{eq:ridge-fw}
    \E \big[ R (\wh f_\lambda) \big]
    &\leq \inf_{\theta \in \R^d} \big\{ R (\theta) + \lambda \norm{\theta}^2 \big\} + 2 m^2\, \frac{\tr [ (\Sigma + \lambda I_d)^{-1} \Sigma]}{n + 1}
      \, ; \\
    \E \big[ R (\wh g_{\lambda, b}) \big]
    &\leq \inf_{\theta \in \R^d} \big\{ R (\theta) + \lambda \norm{\theta}^2 \big\} + 8 m^2\, \frac{\tr [ (\Sigma + \lambda I_d)^{-1} \Sigma]}{n + 1}
      \label{eq:truncated-ridge-bound}
    .
  \end{align}
  In addition, let $\wh b = \max_{1 \leq i \leq n} |Y_i|$ and $\wh g_\lambda (x) = \wh g_{\lambda, \wh b} (x)$.
  Then, we have
  \begin{equation}
    \label{eq:adaptive-truncated-bound}
    \E \big[ R (\wh g_{\lambda}) \big]
    \leq \inf_{\theta \in \R^d} \big\{ R (\theta) + \lambda \norm{\theta}^2 \big\} + 8 b^2\, \frac{\tr [ (\Sigma + \lambda I_d)^{-1} \Sigma]}{n + 1} + \frac{b^2}{n+1}
    .
  \end{equation}
\end{theorem}

The three estimators $\wh f_\lambda, \wh g_{\lambda, b}$ and $\wh g_\lambda$ achieve the same risk bound up to constant factors.
The main differences are:
\begin{itemize}
\item $\wh f_\lambda$
  is more expensive computationally than truncated ridge at prediction time.
  Indeed, truncated ridge only requires to store the vector $\wh \theta_{\lambda'} \in \R^d$, which takes $O (d)$ memory (in the case of infinite- or high-dimensional reproducing kernel Hilbert spaces, using the so-called kernel trick one can replace $d$ by $n$ in all this discussion); computing predictions involves a scalar product, which takes $O (d)$ time.
  In contrast, $\wh f_\lambda$ requires storing both $\wh \theta_{\lambda'}$ and the precision matrix $S_\lambda = (\sum_{i=1}^{n} X_i X_i^\top + \lambda (n+1) I_d)^{-1}$, which takes $O (d^2)$ memory.
  Computation of the leverage $h_{\lambda'} (x) = \innerp{S_\lambda x}{x} / (1+\innerp{S_\lambda x}{x})$ (by Sherman-Morrison's identity) involves computing $\innerp{S_\lambda x}{x}$, which takes $O (d^2)$ time.
\item On the other hand, from a statistical perspective, $\wh g_{\lambda, b}$ requires prior knowledge of the parameter $b$, while $\wh f_\lambda$ is adaptive to $b$.
  This shortcoming of $\wh g_{\lambda, b}$ is addressed by the adaptive truncated estimator $\wh g_{\lambda}$.
\end{itemize}

In fact, one could replace $\tr [ (\Sigma + \lambda I_d)^{-1} \Sigma ]$ by the (smaller) quantity $\E \tr [ (\wh \Sigma_{n+1} + \lambda)^{-1} \wh \Sigma_{n+1} ]$, which is always defined without any assumption on the distribution of $X$.
We stated the bound in terms of the population covariance since it is more explicit.
In any case, it is worth noting that this bound is true under only a second moment assumption on $X$.

We also mention that a remarkable procedure for linear regression in the sequential setting was proposed by Vovk~\cite{vovk2001competitive} and Azoury and Warmuth~\cite{azoury2001relative}.
Combining its regret guarantee with the conversion described in the introduction, the averaged version $\wh g_\lambda$ of this procedure satisfies the following risk bound:
\begin{equation*}
  \E R (\wh g_\lambda)
  \leq \inf_{\theta \in \R^d} \big\{ R (\theta) + \lambda \norm{\theta}^2 \big\} + \frac{b^2 \cdot \log \det (\lambda^{-1} \Sigma + I_d)}{n + 1}
  .
\end{equation*}
Compared to this bound, the guarantees from Theorem~\ref{thm:improper-ridge} allow to replace the log-determinant by a smaller trace term.

\subsection{Aggregation of classifiers and almost exchangeable priors}
\label{sec:transductive}
\label{sec:binaryclass}

In the classic paper \cite{haussler1998sequential}, the following question is presented: Given a finite set $\mathcal F$ comprised of $\{0,1\}$-valued functions and assuming that $y_i \in \{0, 1\}$, how can one demonstrate the optimal sequential regret bound \eqref{eq:def-regret} for a particular loss function with curvature as a function of $n$ and the cardinality of $\mathcal F$? An important particular case of their main result is the squared loss function, which is also the subject of our focus. Another interesting analysis of classes of binary-valued functions in the context of statistical aggregation appears in \cite{rakhlin2017empirical}.
To present their result, we first need some standard definitions.
Let $\mathcal F = \{f_{\theta} : \theta \in \Theta\}$ be a class of $\{0, 1\}$-valued functions with the domain $\mathcal X$. The Vapnik-Chervonenkis (VC) dimension of $\mathcal F$ is the largest integer $d$ such that
\[
\max\limits_{x_1, \ldots, x_d \in \mathcal X}\left|\{(f(x_1), \ldots, f(x_d)): f \in \mathcal F\}\right| = 2^d.
\]
Classes of $\{0, 1\}$-valued functions with a finite VC dimension will be referred to as VC classes. Recall that for the squared loss we have $R(f) = \E[(f(X)-Y)^2]$. The result in \cite[Theorem 3 and Remark 4]{rakhlin2017empirical} can be stated as follows. For any VC class $\mathcal F$ whose VC dimension if equal to $d$ there is an aggregation procedure $\widehat{f}_{\textrm{RST}}$ such that, with probability at least $1 - \delta$,
\begin{equation}
  \label{eq:rakhlinsridharan}
  R(\widehat{f}_{\textrm{RST}})- \inf_{f \in \mathcal F}R(f) \le c\left(\frac{d}{n}\log\left(\frac{en}{d}\right) + \frac{1}{n}\log\left(\frac{1}{\delta}\right)\right),
\end{equation}
where $c > 0$ is some absolute constant provided that $n \ge d$. The primary difference from binary classification is the use of squared loss instead of the standard binary loss. This aligns with the work~\cite{haussler1998sequential}. As a result, one can get the fast for the improper estimator $\widehat{f}_{\textrm{RST}}$ without making any low noise/margin assumptions. Rakhlin, Sridharan and Tsybakov~\cite{rakhlin2017empirical} also present an example of VC class showing that this bound cannot be improved (including the $\log\left(en/d\right)$ term) up to multiplicative constant factors.
That is, at least in some cases, the inequality \eqref{eq:rakhlinsridharan} can be reversed (but with a different absolute constant).

Recent efforts in the field of binary classification have focused on improving the $\log\left(en/d\right)$ factor in the bounds. The authors of \cite{hanneke2016refined} and \cite{zhivotovskiy2018localization} leveraged additional complexity measures of VC classes to capture the refined local structure of the classes. A key complexity measure discussed in these papers is the \emph{star number}, which is defined as the largest integer $s$ such that there exists a subset $S =\{ x_1, \dots , x_s\} \subseteq \mathcal X$ and functions $f_0 , f_1,\dots , f_s \in \mathcal{F}$
such that for all $i = 1, \ldots, s$, $f_0$ and $f_i$ disagree on a unique point in $S$. That is, for all $i = 1, \ldots, s$, we have $\{x_i\} = \{x \in S: f_0(x) \neq f_i(x)\}$. In binary classification, it is shown in \cite{hanneke2016refined, zhivotovskiy2018localization} that the term $\log\left(en/d\right)$ can usually be replaced in the risk bounds with $\log\left(\min\{en/d, s\}\right)$ if the star number of $\mathcal F$ is finite. A natural question arises as to whether the star number, which characterizes the local structure of VC classes, can be applied in the context of bound \eqref{eq:rakhlinsridharan}. It appears that the methods of this paper allow us to do so. In order to provide such a setup, we extend our analysis to handle the so-called \emph{transductive algorithms}.
This approach was proposed by Vapnik~\cite{vapnik1999nature} and has been demonstrated to have useful connections with PAC-Bayesian arguments in subsequent works \cite{catoni2007pacbayes, audibert2004phd}. In the simplest form of the transductive setup, the learner works with an i.i.d sample of size $2n$. Among these points, the first $n$ points have their labels $Y_i$, while for the remaining $n$ points we only observe the objects $X_i$, but not their labels. We are now aiming to prove the high probability upper bound on the empirical excess risk computed with respect to the second part of the sample:
\[
R^{\prime}_{n}(\widehat{f}) - \min\limits_{f \in \F}R^{\prime}_{n}(f) = \frac{1}{n} \bigg( \sum_{i =n+1}^{2n}(\widehat{f}(X_i) - Y_i)^2 - \min\limits_{f \in \mathcal F}\sum_{i =n+1}^{2n}(f(X_i) - Y_i)^2 \bigg),
\]
where the estimator $\widehat{f}$ depends on $n$ labeled instances and all $2n$ unlabeled points. The idea is to replace the population risk $R(f)$ by its empirical counterpart $R^{\prime}_{n}(f)$ computed with respect to the second half of the sample.
Furthermore, for the inverse temperature parameter $\beta > 0$, we redefine the (data-dependent) localized prior as follows:
\begin{equation}
  \pi_{\beta (R_n +R^{\prime}_{n})}(\di \theta) =
  \frac{\exp(-\beta(R_n(f_{\theta}) +
  R^{\prime}_{n}(f_{\theta})))}{\int_{\Theta}\exp(-\beta(R_n(f_{\vartheta})+
  R^{\prime}_{n}(f_{\vartheta})))\pi(\di \vartheta)}.
\end{equation}
The next theorem is a direct extension of Theorem \ref{thm:pac-bayes-q-aggregation} to the transductive setup.
\begin{theorem}[A local bound with almost exchangeable priors]
\label{thm:exchangablepriors}
  Consider the transductive classification setup discussed above. Assume that the boundedness assumption \eqref{eq:boundedassumption} holds. Let $\pi \in \probas (\Theta)$ be any probability distributions that may depend on the sample of size $2n$.
  Let $\fq_{n}$ be the output of the $Q$-aggregation estimator
  defined in \eqref{eq:q-aggregation-estimator-definition} with inverse temperature
  $\beta = c_1n/b^2$. Then, for any $\delta \in (0, 1)$, with probability at least $1 - \delta$, it holds that
  \begin{equation}
    R_n^{\prime}(\fq_{n}) \leq
    \inf_{\gamma \in \probas(\Theta)}\left\{
        \langle \gamma, R^{\prime}_{n}\rangle + \frac{3b^2}{2c_1 n}\kl(\gamma,
        \pi_{-\frac{c_{1} n}{6b^2}(R_n + R^{\prime}_{n})})\right\} + c\frac{b^2 \log(3/\delta)}{n}.
    \end{equation}
  where $c, c_1 > 0$ are universal constants.
\end{theorem}

The fact that our localized prior distribution now depends on the entire sample of size $2n$ makes the setup of Theorem \ref{thm:exchangablepriors} closely related to the idea of \emph{almost exchangeable priors}\footnote{The prior distribution $\pi$  depends on the entire sample of size $2n$, while two part of the sample are exchangeable due to independence, hence the terminology.} introduced in \cite{catoni2007pacbayes,audibert2004phd} and to \emph{conditional mutual information} appearing in the recent works \cite{steinke2020reasoning,grunwald2021pac,haghifam2021towards}.
Our focus is on the aggregation problem with squared loss without the realizability assumption, whereas prior work primarily addressed the standard binary classification setup.

We now demonstrate the application of Theorem \ref{thm:exchangablepriors} in the context of binary classification.
Consider a VC class $\mathcal F$. Given an \iid sample $S_{2n} = \{X_1, \ldots, X_{2n}\}$, we first consider the (random) projection class $\mathcal F|_{S_{2n}} = \{(f(X_1), \ldots, f(X_{2n})): f \in \mathcal F\}$ sample of unlabeled instances and choose the prior distribution $\pi_{S_{2n}}$ to be the uniform measure on $\mathcal{F}_{S_{2n}}$.
When working with our $Q$-Aggregation estimator, we identify $\mathcal{F}|_{S_{2n}}$ with the set $\Theta$. Observe that $\mathcal F|_{S_{2n}}$ is a finite class, so that our uniform distribution is correctly defined.

\begin{proposition}
  \label{thm:transductive-application}
  Consider the transductive setup of Theorem \ref{thm:exchangablepriors} and let $\mathcal F$ be a class of $\{0, 1\}$-valued functions with VC dimension $d$ and star number $s$. Let $\pi = \pi(S_{2n})$ be a uniform measure on the projection class $\mathcal F|_{S_{2n}}$, where $S_{2n}$ denotes the full sample of unlabeled points of size $2n$. Let $\fq_{n}$ be the output of the $Q$-aggregation estimator
  defined in \eqref{eq:q-aggregation-estimator-definition} with the inverse temperature parameter $\beta = c_1n$ and prior $\pi$. Then,
  \[
      R^{\prime}_{n}(\fq_{n}) - \min\limits_{f \in \F}R^{\prime}_{n}(f)
      \leq c\left\{ \frac{d}{n}
        \log\left[ es
          \left(\frac{nR(f^{\star})}{d} + \log\left(\frac{en}{d}\right)\right)
        \right]
      + \frac{\log(4/\delta)}{n}\right\},
    \]
  where $c, c_1>0$ are absolute constants and $R(f^{\star}) = \inf_{f \in \mathcal F}R(f)$.
\end{proposition}
Our result provides the first improvement over the bound \eqref{eq:rakhlinsridharan} in \cite{rakhlin2017empirical}, but only in the transductive setting. An interesting regime is when $R(f^{\star}) \lesssim \frac{s^\kappa}{n}$, where $\kappa > 1$ is some absolute constant. The reason for including $R(f^{\star})$ in the logarithmic factor is that we need to account for the possibility of multiple minimizers of the population risk. This distinguishes our setup from all previous classification bounds that involved the star number, where the low-noise assumption implied the existence of a unique risk minimizer, making the local structure of the class significantly simpler.

\paragraph{Acknowledgments.} Jaouad Mourtada and Nikita Zhivotovskiy acknowledge the support of the France-Berkeley Fund.

\bibliographystyle{abbrv}%

\appendix

\section{Proofs}
\label{sec:proofs}

This appendix contains the proofs of all the results stated in the main text.

\subsection{Basic identities and inequalities}
\label{sec:basic-identities}

In this section, we collect some basic facts related to the $\kl$ divergence,
quadratic loss and the curvature of the empirical the $Q$-aggregation objective
defined in \eqref{eq:q-aggregation-estimator-definition}. These facts are repeatedly
used throughout the proofs of our main theorems stated in
Section~\ref{sec:main-results}.
We will repeatedly use the convex-conjugate duality
of the moment generating function and the $\kl$ divergence written in the following form: for any function $h : \Theta \to \R$ such that the quantities below are well-defined,
\begin{equation}
\label{eq:pacbayes}
\sup\limits_{\rho\in \mathcal{P}(\Theta)}\left\{
  \langle\rho,  h\rangle - \kl(\rho, \pi)
  \right\}
  = \log \int_\Theta \exp\left( h (\theta)\right) \pi (\di \theta).
\end{equation}

In order to obtain localized PAC-Bayes bounds, at some point in our proofs we
will need to replace the difference of localized $\kl$ divergences
$\kl(\rho, \pi_{-f}) - \kl(\gamma, \pi_{-f})$ by the difference $\kl(\rho,\pi)
- \kl(\gamma, \pi)$. The following lemma shows that the price of this
interchange is equal to $\langle \rho - \gamma, f\rangle$, for any function
$f$ used to localize the prior $\pi$.

\begin{lemma}
  \label{lemma:pb-localization}
  For any probability distributions $\rho,\gamma,\pi \in \mathcal{P}(\Theta)$
  such that $\rho \ll \pi$ and $\gamma \ll \pi$, and for any function
  $f : \Theta \to \R$ we have
  \begin{equation}
    \kl(\rho, \pi_{-f}) - \kl(\gamma, \pi_{-f})
    =
    \kl(\rho, \pi) - \kl(\gamma, \pi)
    +
    \langle \rho - \gamma, f \rangle.
  \end{equation}
\end{lemma}

\begin{proof}
  A direct computation yields
  \begin{align}
    \kl(\gamma, \pi_{-f})
    &= %
      \int_\Theta \log \bigg(\frac{\di\gamma}{\di \pi_{-f}}(\theta)\bigg) \gamma (\di \theta)
    \\
    &=
      \int_\Theta
      \log \left(
          \frac{\di \gamma}{\di\pi}(\theta)
          \cdot
          \frac
          {\langle \pi, \exp(-f)\rangle}
          {\exp(-f(\theta))}
      \right)
      \gamma (\di \theta)
    \\
    &=
    \kl(\gamma, \pi) + \langle \gamma, f \rangle
      + \log \langle \pi, \exp(-f)\rangle .
    \label{eq:KL-localized-gamma}
  \end{align}
  Similarly, it holds that
  \begin{equation}
    \label{eq:KL-localized-rho}
    \kl(\rho, \pi_{-f})
    =
    \kl(\rho, \pi) + \langle \rho, f \rangle
    + \log \langle \pi, \exp (-f)\rangle .
  \end{equation}
  Subtracting \eqref{eq:KL-localized-rho} from \eqref{eq:KL-localized-gamma}
  concludes the proof of this lemma.
\end{proof}

Next, we recall some basic properties of the squared loss. Define
\begin{equation}
    \label{eq:empirical-variance-definition}
    V_{n}(\rho) = \langle \rho, R_{n} \rangle - R_{n}(\rho) = \langle \rho, \|f - f_{\rho}\|_{n}^{2}\rangle
    = \int_\Theta \|f_{\theta} - f_{\rho}\|_{n}^{2} \, \rho (\di \theta),
\end{equation}
where we remind that $R_{n}$ is defined in \eqref{eq:empirical-risk}.
Then, for any $\gamma \in \probas(\Theta)$ we have
\begin{equation}
    \label{eq:variance-identity}
    V_{n}(\rho) + \|f_{\rho} - f_{\gamma}\|_{n}^{2} = \langle \rho, \|f - f_{\gamma}\|_{n}^{2}\rangle.
\end{equation}
The two identities stated above are repeatedly applied in the rest of the proofs without explicitly referring to this section.

The following lemma highlights the main property of the $Q$-aggregation estimator that distinguishes its analysis from the exponential weights.
This lemma bears similarity with the analysis using offset Rademacher complexities \cite{liang2015offset, zhivotovskiy2018localization, vijaykumar2021localization, kanade2022exponential}, where the curvature of the loss function allows to obtain a certain ``complexity regularizing'' term.
However, the term obtained in the lemma stated below is different from the one present in the offset Rademacher complexity analysis, which would correspond to a term proportional to $-\|f_{\widehat{\rho}} - f_{\gamma}\|_{n}^{2}$.

\begin{lemma}
  \label{lemma:q-aggregation-quadratic-loss-negative-term}
  Let $\pi \in \probas(\Theta)$ and $\beta > 0$ be arbitrary.
  Let $\widehat{\rho} = \pq_{n}(\beta, \pi)$ be the minimizer of the $Q$-aggregation objective defined in \eqref{eq:q-aggregation-estimator-definition}.
  Then, for any $\gamma \in \probas (\Theta)$,
  the following deterministic inequality holds:
  \begin{equation}
    R_{n}(\widehat{\rho}) - R_{n}(\gamma)
    - \frac{1}{\beta}\kl(\gamma, \pi)
    + \frac{1}{\beta}\kl(\widehat{\rho}, \pi)
    \leq
    -\frac{1}{2}\langle \widehat{\rho}, \|f - f_{\gamma}\|_{n}^{2} \rangle
    + \frac{1}{2}V_{n}(\gamma).
  \end{equation}
\end{lemma}

\begin{proof}
Let $Q_{n}(\rho) = \frac{1}{2}R_{n}(\rho) + \frac{1}{2}\langle \rho, R_{n}\rangle + \frac{1}{\beta}\kl(\rho, \pi)$. Thus, $\widehat{\rho}$ is the unique minimizer of $Q_{n}$.
Observe that
\begin{align}
  R_{n}(\widehat{\rho}) - R_{n}(\gamma)
  =
  \left(
    Q_{n}(\widehat{\rho})
    -
    Q_{n}(\gamma)
  \right)
  - \frac{1}{2}V_{n}(\widehat{\rho})
  + \frac{1}{2}V_{n}(\gamma)
  - \frac{1}{\beta}\kl(\widehat{\rho}, \pi)
  + \frac{1}{\beta}\kl(\gamma, \pi).
  \label{eq:negative-term-lemma-decomposition}
\end{align}
Strong convexity of the squared loss combined with the
optimality of $\widehat{\rho}$ for the objective $Q_{n}$ yields:
\begin{align}
    Q_{n}(\widehat{\rho})
    -
    Q_{n}(\gamma)
    \leq
    - \frac{1}{2}\|f_{\widehat{\rho}} - f_{\gamma}\|_{n}^{2}.
\end{align}
Plugging the above into \eqref{eq:negative-term-lemma-decomposition}
yields
\begin{align}
  &R_{n}(\widehat{\rho}) - R_{n}(\gamma)
  \\
  &\leq
  -\frac{1}{2}\left(
    \|f_{\widehat{\rho}} - f_{\gamma}\|_{n}^{2}
    +
    V_{n}(\widehat{\rho})
  \right)
  + \frac{1}{2}V_{n}(\gamma)
  - \frac{1}{\beta}\kl(\widehat{\rho}, \pi)
  + \frac{1}{\beta}\kl(\gamma, \pi).
  \\
  &=
  -\frac{1}{2}\langle \widehat{\rho}, \|f - f_{\gamma}\|_{n}^{2} \rangle
  + \frac{1}{2}V_{n}(\gamma)
  - \frac{1}{\beta}\kl(\widehat{\rho}, \pi)
  + \frac{1}{\beta}\kl(\gamma, \pi),
\end{align}
where the final identity follows from the variance identity
\eqref{eq:variance-identity}.
\end{proof}

The distribution $\pew_{n}$ defined in \eqref{eq:exp-weights-estimator-definition}
admits a closed form analytic expression. This considerably simplifies the
analysis of the exponential weights estimator. In contrast, there is no simple
closed form expression for the distribution $\pq$ defined in
\eqref{eq:q-aggregation-estimator-definition}. Thus, in general, the analysis
of the $Q$-aggregation estimator poses more challenges than the exponential
weights estimator. A priori, $\pq_{n}$ is an arbitrary probability distribution on
$\Theta$, and therefore it takes values in the infinite-dimensional
space $\probas (\Theta)$ whenever $\Theta$ is infinite-dimensional.
Drawing similarities with the \emph{representer theorem} and the \emph{kernel
trick} commonly employed in
the context of reproducing kernel Hilbert spaces \cite{aronszajn1950theory,steinwart2008svm}, we obtain the
following result proved in Section~\ref{sec:proof-representer}. It states that $\pq_{n}$ lies in a
data-dependent finite-dimensional exponential family of distributions.
This characterization is particularly useful when the prior distribution $\pi$
is Gaussian as we shall later explain in Section~\ref{sec:gauss-priors}.

\begin{proposition}%
  \label{prop:representer}
  There exist (data-dependent) constants $(a_{i})_{1\leq i \leq n} \in \R^{n}$, $(b_{i})_{1 \leq i \leq n} \in \R^{n}$ and $c_0 \in \R$ such that, for $\pi$-almost all $\theta \in \Theta$,
  \begin{equation*}
    \frac{\di \pq_{n}}{\di \pi} (\theta)
    = \exp \bigg\{ \sum_{i=1}^n a_{i} f_{\theta} (X_i)^2 + \sum_{i=1}^n b_{i} f_\theta (X_i) + c_0 \bigg\}
    .
  \end{equation*}
\end{proposition}

\subsection{Proof of Proposition~\ref{prop:representer}}
\label{sec:proof-representer}
  By definition of the $Q$-aggregation estimator, we have $\wh \rho_n = \argmin_{\rho} \{ \beta  Q_{n} (\rho) + \kll{\rho}{\pi} \}$, with
  \begin{align*}
    Q_{n} (\rho)
    &= (1-\alpha) \innerp{\rho}{R_n} + \alpha R_n (\rho) \\
    &= \frac{1-\alpha}{n} \sum_{i=1}^n \int_\Theta (f_\theta (X_i) - Y_i)^2 \rho (\di \theta)
      + \frac{\alpha}{n} \sum_{i=1}^n \bigg( \int_\Theta f_\theta (X_i) \rho (\di \theta) - Y_i \bigg)^2 \\
    &= \Phi \bigg\{ \bigg( \int_\Theta f_\theta (X_i)^2 \rho (\di \theta) \bigg)_{1 \leq i \leq n} , \bigg( \int_\Theta f_\theta (X_i) \rho (\di \theta) \bigg)_{1 \leq i \leq n} \bigg\} \,,
  \end{align*}
  where $\Phi : \R^{n} \times \R^n \to \R$ is defined by:
  \begin{equation*}
    \Phi \big\{ (u_{i})_{1 \leq i \leq n}, (v_i)_{1 \leq i \leq n} \big\}
    = \frac{1-\alpha}{n} \sum_{i=1}^n \big(  u_i - 2 v_i Y_i + Y_i^2 \big) + \frac{\alpha}{n} \sum_{i=1}^n (v_i - Y_i)^2
    .
  \end{equation*}
  Now, let $u_i = \int_{\Theta} f_\theta (X_i)^2 \wh \rho_n (\di \theta)$ and $v_i = \int_{\Theta} f_\theta (X_i) \wh \rho_n (\di \theta)$ for $i= 1, \dots, n$.
  Define the set:
  \begin{equation*}
    \Delta
    = \bigg\{ \rho \in \probas (\Theta) : \int_{\Theta} f_\theta (X_i)^2 \rho (\di \theta) = u_i, \ \int_{\Theta} f_\theta (X_i) \rho (\di \theta) = v_i , \ 1 \leq i \leq n \bigg\}
    .
  \end{equation*}
  Since $\wh \rho_n \in \Delta$, we have $\wh \rho_n = \argmin_{\rho \in \Delta} \{ \beta Q_{n} (\rho) + \kll{\rho}{\pi} \} = \argmin_{\rho \in \Delta} \kll{\rho}{\pi}$ since $Q (\rho) = Q (\wh \rho_n)$ for $\rho \in \Delta$ by the above.
  The conclusion then follows from the maximum entropy theorem; see, \eg \cite[Theorem~12.1.1,~p.~410]{cover2006elements}.
\hfill\qed

\subsection{Proof of Theorem~\ref{thm:localized-exp-weights-gaussian-model}}
\label{sec:proof-fixed-design-exp-weights}

Under the Gaussian noise assumption, we have
$$
  Y \sim \mathcal{N}(f^{\star}, \sigma^{2}I_{n}),
$$
where with a slight abuse of notation, $f^{\star}$ denotes the vector $(f^{\star}(X_{1}), \dots, f^{\star}(X_{n})) \in \R^{n}$.

We now recall some basic facts about Stein's Unbiased Risk Estimator (SURE).
Let $\widehat{f} : \R^{n} \to \R^{n}$ be any estimator of the
mean vector $f^{\star}$ as a function of the observations vector $Y$.
Then, SURE for the estimator $\widehat{f}$ is given by the identity
\begin{equation}
  \label{eq:sure-definition}
  \mathrm{SURE}(\widehat{f}, Y)
  =
  \|\widehat{f}(Y) - Y\|_{2}^{2}
  + 2\sigma^{2}\sum_{i=1}^{n}\frac{\partial}{\partial
  Y_{i}}\widehat{f}(Y)_{i}
  - n\sigma^{2},
\end{equation}
where $\widehat{f}(Y)_{i}$ denotes the $i$-th component of the $n$-dimensional vector $\widehat{f}(Y)$.
It suffices to assume the existence, continuity, and integrability of the above partial derivatives so that a simple Gaussian integration by parts argument yields
\begin{equation}
  \label{eq:sure-is-unbiased}
  \E_{Y}\left[\|\widehat{f} - f^{\star}\|_{2}^{2}\right]
  =
  \E_{Y}\left[
    \mathrm{SURE}(\widehat{f}, Y)
  \right].
\end{equation}

For the exponential weights estimator $\few_{n} = \langle \pew_{n}(\beta, \pi), f \rangle$, SURE admits a closed-form expression. 
To simplify the notation, in what follows let $\widehat{\rho} = \pew_{n}(\beta, \pi)$ and $\widehat{f} = \few_{n}$. Then, we have
\begin{equation}
  \widehat{\rho}(\di\theta)
  =
  \frac{\exp(-\beta R_{n}(f_{\theta}))\pi(d \theta)}
  {\int_{\Theta} \exp(-\beta R_{n}(f_{\theta}))\pi(d \theta)}.
\end{equation}
A direct computation of the partial derivatives yields
\begin{equation}
  \frac{\partial}{\partial Y_{i}}(\widehat{f}(Y))_{i}
  =
  \frac{2\beta}{n}\int_{\Theta}\left(
    (f_{\theta})_{i} - (\widehat{f})_{i}
  \right)^{2}
  \widehat{\rho}(\di\theta)
\end{equation}
and hence,
\begin{equation}
  \label{eq:sure-for-exponential-weights}
  \frac{1}{n}
  \mathrm{SURE}(\widehat{f}, Y)
  =
  \langle \widehat{\rho}, R_{n} \rangle
  + \bigg(\frac{4\beta\sigma^{2}}{n} - 1 \bigg)
  V_{n}(\widehat{\rho})
  - \sigma^{2}.
\end{equation}
In previous analyses of the exponential weights estimator
\cite{leung2006information,dalalyan2008aggregation},
the second term in the right hand side of the
above identity is dropped by taking
$\beta \leq \frac{n}{4\sigma^{2}}$ .
This is the part where our analysis departs from the proofs that yield global bounds: taking a smaller $\beta$ and keeping this term---that now
becomes strictly negative---is crucial in order to prove the localized bound stated in Theorem~\ref{thm:localized-exp-weights-gaussian-model}.
We now proceed to finish the proof of this theorem.

The constraints put on $\beta$ in the theorem statement guarantee that
$\frac{4\beta\sigma^{2}}{n} - 1 \leq -\frac{1}{2}$.
Hence, we may bound the SURE estimator \eqref{eq:sure-for-exponential-weights}
by
  \begin{equation}
    \label{eq:sure-bound}
    \frac{1}{n}
    \mathrm{SURE}(\widehat{f}, Y)
    \leq
    \langle \widehat{\rho}, R_{n} \rangle
    - \frac{1}{2}
    V_{n}(\widehat{\rho})
    - \sigma^{2}.
  \end{equation}
  Adding $V_{n}(\widehat{\rho}) + \sigma^{2}$ to both sides of the above inequality and recalling that   $\widehat{\rho}$ minimizes the $\kl$-divergence penalized
  empirical risk minimization objective
  \eqref{eq:exp-weights-estimator-definition}, for any $\gamma \in
  \mathcal{P}(\Theta)$ it holds that
  \begin{align}
    &\frac{1}{n}
    \mathrm{SURE}(\widehat{f}, Y)
    + \frac{1}{2}V_{n}(\widehat{\rho})
    + \sigma^{2}
    \\
    &\leq
    \langle \widehat{\rho}, R_{n} \rangle
    + \frac{1}{\beta}\kl(\rho, \pi)
    - \frac{1}{\beta}\kl(\rho, \pi)
    \\
    &\leq
    \langle \gamma, R_{n} \rangle
    + \frac{1}{\beta}\kl(\gamma, \pi)
    - \frac{1}{\beta}\kl(\rho, \pi)
    \\
    &=
    \langle \gamma, R_{n} \rangle
    + \frac{1}{\beta}\kl(\gamma, \pi_{-\frac{\beta}{2}R})
    - \frac{1}{\beta}\kl(\rho, \pi_{-\frac{\beta}{2}R})
    + \frac{1}{2}\langle \widehat{\rho} - \gamma, R \rangle
    \\
    &\leq
    \langle \gamma, R_{n} \rangle
    + \frac{1}{\beta}\kl(\gamma, \pi_{-\frac{\beta}{2}R})
    + \frac{1}{2}\langle \widehat{\rho} - \gamma, R \rangle,
    \label{eq:exp-weights-fixed-design-localization-proof-almost-final-step}
  \end{align}
  where the equality step follows from Lemma~\ref{lemma:pb-localization}.
  Now observe that
  \begin{equation}
  \langle \widehat{\rho}, R \rangle = \langle \widehat{\rho}, \|f - f^{\star}\|_{n}^{2}\rangle + \sigma^{2}
  = \|\widehat{f} - f^{\star}\|_{n}^{2} + V_{n}(\widehat{\rho}) + \sigma^{2}.
  \end{equation}
  Plugging the above identity into
  \eqref{eq:exp-weights-fixed-design-localization-proof-almost-final-step} and
  rearranging yields
  \begin{align}
    \frac{1}{n}
    \mathrm{SURE}(\widehat{f}, Y)
    -
    \frac{1}{2}\|\widehat{f} - f^{\star}\|_{n}^{2}
    \leq
    \langle \gamma, R_{n} - \frac{1}{2} R \rangle
    + \frac{1}{\beta}\kl(\gamma, \pi_{-\frac{\beta}{2}R})
    - \frac{1}{2}\sigma^{2}.
  \end{align}
  Taking expectations on both sides and using the fact that $\gamma$ does not
  depend on the sample $Y = (Y_{1},\dots,Y_{n})^{\top}$, we have
  \begin{align}
    \E_{Y}\left[
      \frac{1}{n}
      \mathrm{SURE}(\widehat{f}, Y)
    \right]
    -
    \frac{1}{2}
    \E_{Y}\left[
      \|\widehat{f} - f^{\star}\|_{n}^{2}
    \right]
    &\leq
    \frac{1}{2}
      \langle \gamma, R \rangle
    + \frac{1}{\beta}\kl(\gamma, \pi_{-\frac{\beta}{2}R})
    - \frac{1}{2}\sigma^{2}
    \\
    &=
    \frac{1}{2}
    \E_{\theta \sim \gamma}\left[
      \|f_{\theta} - f^{\star}\|_{n}^{2}
    \right]
    + \frac{1}{\beta}\kl(\gamma, \pi_{-\frac{\beta}{2}R}).
  \end{align}
  Finally, because SURE is unbiased \eqref{eq:sure-is-unbiased},
  it follows that
\[
    \E_{Y}\left[\|\widehat{f} - f^{\star}\|_{n}^{2}\right]
\leq
    \E_{\theta \sim \gamma}\left[
      \|f_{\theta} - f^{\star}\|_{n}^{2}
    \right]
    + \frac{2}{\beta}\kl(\gamma, \pi_{-\frac{\beta}{2}R}).
\]
  Recalling that the choice of $\gamma \in \mathcal{P}(\Theta)$ is arbitrary {and using the inequality \eqref{eq:localvsglobalbound}}, we finish the proof.\hfill\qed

\subsection{Proof of Theorem~\ref{thm:localized-fixed-design-q-aggregation}}
\label{sec:proof-q-aggregation-fixed-design}
  To simplify the notation, let $\widehat{f} = \fq_{n}$ and $\widehat{\rho} = \pq_{n}(\beta, \pi)$.
  Moreover, we identify all the functions $f : \mathcal{X} \to \R$ as $n$-dimensional vectors $(f(X_{1}), \dots, f(X_{n}))^{\top} \in \R^{n}$.
  Let $\gamma$ be the minimizer of
  $\gamma' \mapsto
      \langle \gamma', \|f - f^{\star}\|_{n}^{2} \rangle
      + \frac{3}{2\beta}
      \kl(\gamma', \pi_{-\frac{\beta}{3}R}).
  $
  Then, we have
  \begin{align}
    &\|f_{\widehat{\rho}} - f^{\star}\|_{n}^{2}
    -
    \inf_{\gamma' \in \mathcal{P}(\Theta)}
    \left\{
       \langle \gamma', \|f - f^{\star}\|_{n}^{2} \rangle
      + \frac{3}{2\beta}
      \kl(\gamma', \pi_{-\frac{\beta}{3}R})
    \right\}
    \\
    &=
      \|f_{\widehat{\rho}} - f^{\star}\|_{n}^{2}
      -
       \langle \gamma, \|f - f^{\star}\|_{n}^{2} \rangle
      - \frac{3}{2\beta}
      \kl(\gamma, \pi_{-\frac{\beta}{3}R})
    \\
    &=
      \|f_{\widehat{\rho}} - f^{\star}\|_{n}^{2}
      -
      \|f_{\gamma} - f^{\star}\|_{n}^{2}
      - V_{n}(\gamma)
      - \frac{3}{2\beta}
      \kl(\gamma, \pi_{-\frac{\beta}{3}R})
    \\
    &=
      R_{n}(\widehat{\rho})
      - R_{n}(\widehat{\gamma})
      +\frac{2}{n}\langle f_{\widehat{\rho}} - f_{\gamma}, Z \rangle
      - V_{n}(\gamma)
      - \frac{3}{2\beta}
      \kl(\gamma, \pi_{-\frac{\beta}{3}R})
    \\
    &=
    R_{n}(\widehat{\rho})
    + \frac{3}{2\beta} \kl(\widehat{\rho}, \pi_{-\frac{\beta}{3}R})
    - R_{n}(\widehat{\gamma})
    - \frac{3}{2\beta}\kl(\gamma, \pi_{-\frac{\beta}{3}R})
    \\
    &\quad\quad
      +\frac{2}{n}\langle f_{\widehat{\rho}} - f_{\gamma}, Z \rangle
      - V_{n}(\gamma)
      - \frac{3}{2\beta}
      \kl(\widehat{\rho}, \pi_{-\frac{\beta}{3}R}).
      \label{eq:fixed-design-proof-q-aggregation-intermediate-step}
  \end{align}
  Next, applying Lemma~\ref{lemma:pb-localization} we have
  \begin{align}
    &\frac{3}{2\beta} \kl(\widehat{\rho}, \pi_{-\frac{\beta}{3}R})
     - \frac{3}{2\beta}\kl(\gamma, \pi_{-\frac{\beta}{3}R})
    \\
    &=
     \frac{3}{2\beta} \kl(\widehat{\rho}, \pi)
     - \frac{3}{2\beta}\kl(\gamma, \pi)
     + \frac{3}{2\beta}\langle \widehat{\rho} - \gamma,
        \frac{\beta}{3}R\rangle
    \\
    &=
     \frac{3}{2\beta} \kl(\widehat{\rho}, \pi)
     - \frac{3}{2\beta}\kl(\gamma, \pi)
     + \frac{1}{2}R(\widehat{\rho}) + \frac{1}{2}V_{n}(\widehat{\rho})
     - \frac{1}{2}R(\gamma) - \frac{1}{2}V_{n}(\gamma)
      \\
    &=
     \frac{3}{2\beta} \kl(\widehat{\rho}, \pi)
     - \frac{3}{2\beta}\kl(\gamma, \pi)
     + \frac{1}{2}R_n(\widehat{\rho}) + \frac{1}{2}V_{n}(\widehat{\rho})
     - \frac{1}{2}R_n(\gamma) - \frac{1}{2}V_{n}(\gamma)
     + \frac{1}{n}\langle f_{\widehat{\rho}} - f_{\gamma}, Z \rangle.
  \end{align}
  Plugging the above identity into
  \eqref{eq:fixed-design-proof-q-aggregation-intermediate-step} results in the
  following equality
  \begin{align}
    &\|f_{\widehat{\rho}} - f^{\star}\|_{n}^{2}
    -
    \inf_{\gamma' \in \mathcal{P}(\Theta)}
    \left\{
       \langle \gamma', \|f - f^{\star}\|_{n}^{2} \rangle
      + \frac{3}{2\beta}
      \kl(\gamma', \pi_{-\frac{\beta}{3}R})
    \right\}
    \\
    &=
    \frac{3}{2}\left(
      R_{n}(\widehat{\rho})
      + \frac{1}{\beta} \kl(\widehat{\rho}, \pi)
      - R_{n}(\widehat{\gamma})
      - \frac{1}{\beta}\kl(\gamma, \pi)
    \right)
    \label{eq:q-aggregation-fixed-design-curvature-term}
    \\
    &\quad\quad
      +\frac{3}{n}\langle f_{\widehat{\rho}} - f_{\gamma}, Z \rangle
      - \frac{3}{2}V_{n}(\gamma)
      + \frac{1}{2}V_{n}(\widehat{\rho})
      - \frac{3}{2\beta}
      \kl(\widehat{\rho}, \pi_{-\frac{\beta}{3}R}).
  \end{align}
  Bounding the term \eqref{eq:q-aggregation-fixed-design-curvature-term} via
  the deterministic inequality proved in
  Lemma~\ref{lemma:q-aggregation-quadratic-loss-negative-term},
  the right-hand side of the above identity is upper bounded by
  \begin{align}
    &-\frac{3}{4}\langle\widehat{\rho}, \|f - f_{\gamma}\|_{n}^{2}\rangle
    + \frac{3}{4}V_{n}(\gamma)
      +\frac{3}{n}\langle f_{\widehat{\rho}} - f_{\gamma}, Z \rangle
      - \frac{3}{2}V_{n}(\gamma)
      + \frac{1}{2}V_{n}(\widehat{\rho})
      - \frac{3}{2\beta}
      \kl(\widehat{\rho}, \pi_{-\frac{\beta}{3}R})
    \\
    &\leq
    -\frac{3}{4}\langle\widehat{\rho}, \|f - f_{\gamma}\|_{n}^{2}\rangle
      +\frac{3}{n}\langle f_{\widehat{\rho}} - f_{\gamma}, Z \rangle
      + \frac{1}{2}V_{n}(\widehat{\rho})
      - \frac{3}{2\beta}
      \kl(\widehat{\rho}, \pi_{-\frac{\beta}{3}R})
    \\
    &\leq
    -\frac{1}{4}\langle\widehat{\rho}, \|f - f_{\gamma}\|_{n}^{2}\rangle
      +\frac{3}{n}\langle f_{\widehat{\rho}} - f_{\gamma}, Z \rangle
      - \frac{3}{2\beta}
      \kl(\widehat{\rho}, \pi_{-\frac{\beta}{3}R}),
  \end{align}
  where in the final inequality we used the fact that
  $\langle \widehat{\rho}, \|f - f_{\gamma}\|_{n}^{2} \rangle
  \geq V_{n}(\widehat{\rho})$.
  Putting everything together results in the following deterministic
  inequality:
  \begin{align}
    &\frac{2\beta}{3}
    \|f_{\widehat{\rho}} - f^{\star}\|_{n}^{2}
    -
    \frac{2\beta}{3}
    \inf_{\gamma' \in \mathcal{P}(\Theta)}
    \left\{
        \langle \gamma', \|f - f^{\star}\|_{n}^{2} \rangle
      + \frac{3}{2\beta}
      \kl(\gamma', \pi_{-\frac{\beta}{3}R})
    \right\}
    \\
    &\leq
    -\frac{\beta}{6}\langle\widehat{\rho}, \|f - f_{\gamma}\|_{n}^{2}\rangle
      +\frac{2\beta}{n}\langle f_{\widehat{\rho}} - f_{\gamma}, Z \rangle
      - \kl(\widehat{\rho}, \pi_{-\frac{\beta}{3}R})
    \\
    &\leq
    \sup_{\rho \in \mathcal{P}(\Theta)}\left\{
    -\frac{\beta}{6}\langle \rho, \|f - f_{\gamma}\|_{n}^{2}\rangle
      +\frac{2\beta}{n}\langle f_{\rho} - f_{\gamma}, Z \rangle
      - \kl(\rho, \pi_{-\frac{\beta}{3}R})
    \right\}
    \\
    &=
    \log\left(
      \E_{\theta \sim \pi_{-\frac{\beta}{3}R}}
      \left[
        \exp\left(
          -\frac{\beta}{6}\|f_{\theta} - f_{\gamma}\|_{n}^{2}
          + \frac{2\beta}{n}\langle f_{\theta} - f_{\gamma}, Z \rangle
        \right)
      \right]
    \right) = U,
    \label{eq:q-aggregation-fixed-design-final-deterministic-bound}
  \end{align}
  where the last line follows from \eqref{eq:pacbayes}.
  We will now show that the right-hand side of the above inequality is small
  with high probability by bounding its moment-generating function.
  Indeed, we have
  \begin{align}
    \E_{Z}\left[\exp(U)\right]
    &=
    \E_{Z}
    \E_{\theta \sim \pi_{-\frac{\beta}{3}R}}
    \left[
      \exp\left(
        -\frac{\beta}{6}\|f_{\theta} - f_{\gamma}\|_{n}^{2}
        + \frac{2\beta}{n}\langle f_{\theta} - f_{\gamma}, Z \rangle
      \right)
    \right]
    \\
    &=
    \E_{\theta \sim \pi_{-\frac{\beta}{3}R}}
    \E_{Z}
    \left[
      \exp\left(
        -\frac{\beta}{6}\|f_{\theta} - f_{\gamma}\|_{n}^{2}
        + \frac{2\beta}{n}\langle f_{\theta} - f_{\gamma}, Z \rangle
      \right)
    \right]
    \\
    &=
    \E_{\theta \sim \pi_{-\frac{\beta}{3}R}}
    \left[
    \exp\left(
     -\frac{\beta}{6}\|f_{\theta} - f_{\gamma}\|_{n}^{2}
    \right)
    \E_{Z}
    \left[
      \exp\left(
        \frac{2\beta}{n}\langle f_{\theta} - f_{\gamma}, Z \rangle
      \right)
    \right]
  \right]
    \\
    &=
    \E_{\theta \sim \pi_{-\frac{\beta}{3}R}}
    \left[
    \exp\left(
     -\frac{\beta}{6}\|f_{\theta} - f_{\gamma}\|_{n}^{2}
    \right)
    \prod_{i=1}^{n}
    \E_{Z_{i}}
    \left[
      \exp\left(
        \frac{2\beta}{n}(f_{\theta}(X_{i}) - f_{\gamma}(X_{i}))Z_{i}
      \right)
    \right]
    \right],
  \end{align}
  where the last line follows by the independence of the noise variables
  $Z_{1}, \dots, Z_{n}$. Further, by the sub-Gaussianity assumption on the
  noise we have
  \begin{align}
    \E_{Z}\left[\exp(U)\right]
    &\leq
    \E_{\theta \sim \pi_{-\frac{\beta}{3}R}}
    \left[
      \exp\left(
       -\frac{\beta}{6}\|f_{\theta} - f_{\gamma}\|_{n}^{2}
      \right)
      \prod_{i=1}^{n}
      \exp\left(
        \frac{2\beta^2(f_{\theta}(X_{i}) - f_{\gamma}(X_{i}))^{2}\sigma^{2}}{n^2}
      \right)
    \right]
    \\
    &=
    \E_{\theta \sim \pi_{-\frac{\beta}{3}R}}
    \left[
      \exp\left(
       \left[\frac{2\beta^{2}\sigma^{2}}{n}
       - \frac{\beta}{6}\right]
       \|f_{\theta} - f_{\gamma}\|_{n}^{2}
      \right)
    \right].
  \end{align}
  In particular, for any $\beta \in (0, \frac{n}{12\sigma^{2}}]$
  it holds that $\E_{Z}[\exp(U)] \leq 1$.
  Hence, by Markov's inequality and the bound
  \eqref{eq:q-aggregation-fixed-design-final-deterministic-bound},
  for any $\delta \in (0,1)$, with probability at least $1-\delta$, it holds that
  \begin{equation}
    \frac{2\beta}{3}
    \|f_{\widehat{\rho}} - f^{\star}\|_{n}^{2}
    -
    \frac{2\beta}{3}
    \inf_{\gamma' \in \mathcal{P}(\Theta)}
    \left\{
        \langle \gamma', \|f - f^{\star}\|_{n}^{2} \rangle
      + \frac{3}{2\beta}
      \kl(\gamma', \pi_{-\frac{\beta}{3}R})
    \right\}
    \leq U
    \leq \log(1/\delta).
  \end{equation}
  By re-scaling the above inequality by $3/(2\beta)$ {and using the inequality \eqref{eq:localvsglobalbound}}, we finish the
  proof.
  \hfill\qed

\subsection{Proof of Theorem~\ref{thm:pac-bayes-q-aggregation}}
\label{sec:proof-random-design-localization}

Let $P_{X}$ be the marginal distribution of the covariates $X$ and for any $f \in L_{2}(P_{X})$ introduce the notation $$\|f\|_{2}^{2} = \|f\|_{L_{2}(P_{X})}^{2} = \E_{X\sim P_{X}}[f(X)^2].$$
Then, for any $\gamma \in \probas(\Theta)$ let $V(\gamma) = \langle \gamma, \|f - f_{\gamma}\|_{2}^{2} \rangle$; note that this is a population counterpart to the empirical variance term $V_{n}$ defined in \eqref{eq:empirical-variance-definition}.

Because our proof involves classical empirical process arguments based on symmetrization and contraction, it is convenient to introduce notation commonly used in the empirical processes theory.
Recall that $P$ is the data-generating probability measure and let $P_{n}$ be its empirical counterpart; that is, $P_{n} = n^{-1}\sum_{i=1}^{n}\delta_{(X_{i},Y_{i})}$, where $\delta_{(X_{i}, Y_{i})}$ denotes the Dirac measure concentrated on the $i$-th data point. Then, for
any function $h : \mathcal{X} \times [-b, b] \to \R$, we set
$$
  Ph = \E_{(X,Y) \sim P}[h(X,Y)]
  \quad\text{and}\quad
  P_{n}h = \frac{1}{n}\sum_{i=1}^{n}h(X_{i},Y_{i}).
$$
In particular, for any probability measure $\gamma \in \probas(\Theta)$ we introduce the function $\ell_{\gamma}(X,Y) = (f_{\gamma}(X) - Y)^{2}$. Then, we have
$$
    P\ell_{\gamma} = R(\gamma)\quad\text{and}\quad P_{n}\ell_{\gamma} = R_{n}(\gamma).
$$
Finally, to simplify the notation let $\widehat{\rho} = \pq_{n}(\beta, \pi)$.
We are now ready to proceed with the proof of the theorem.
Let $\gamma$ be the minimizer of $R(\cdot) + \frac{3}{4}V(\cdot) +
\frac{3}{2\beta}\kl(\cdot, \pi_{-\frac{\beta}{3}R})$ over $\probas (\Theta)$.
Then, we have
\begin{align}
  &R(\widehat{\rho}) -
  \inf_{\gamma' \in \probas (\Theta)}
  \left\{
    \langle \gamma', R \rangle
    + \frac{3}{2\beta} \kl(\gamma', \pi_{-\frac{\beta}{3}R})
  \right\}
  \\
  &\leq
  R(\widehat{\rho}) -
  \inf_{\gamma' \in \probas (\Theta)}
  \left\{
    R(\gamma') + \frac{3}{4}V(\gamma')
    + \frac{3}{2\beta} \kl(\gamma', \pi_{-\frac{\beta}{3}R})
  \right\}
  \\
  &=R(\widehat{\rho})
  - R(\gamma)
  - \frac{3}{4}V(\gamma)
  - \frac{3}{2\beta} \kl(\gamma, \pi_{-\frac{\beta}{3}R})
  \\
  &=
  P(\ell_{\widehat{\rho}} - \ell_{\gamma})
  - \frac{3}{4}V(\gamma)
    + \frac{3}{2\beta} \Big\{
    \kl(\widehat{\rho}, \pi_{-\frac{\beta}{3}R})  - \kl(\gamma, \pi_{-\frac{\beta}{3}R}) \Big\}
  - \frac{3}{2\beta} \kl(\widehat{\rho}, \pi_{-\frac{\beta}{3}R})
  \\
  &=
  P(\ell_{\widehat{\rho}} - \ell_{\gamma})
  - \frac{3}{4}V(\gamma)
    + \frac{3}{2\beta} \Big\{ \kl(\widehat{\rho}, \pi) - \kl(\gamma, \pi) \Big\}
    + \frac{1}{2}\langle \widehat{\rho} - \gamma, R \rangle
  - \frac{3}{2\beta} \kl(\widehat{\rho}, \pi_{-\frac{\beta}{2}R})
  \\
  &=
  \frac{3}{2}P(\ell_{\widehat{\rho}} - \ell_{\gamma})
  - \frac{5}{4}V(\gamma)
  + \frac{1}{2}V(\widehat{\rho})
  - \frac{3}{2\beta} \kl(\gamma, \pi)
  + \frac{3}{2\beta} \kl(\widehat{\rho}, \pi)
  - \frac{3}{2\beta} \kl(\widehat{\rho}, \pi_{-\frac{\beta}{2}R}),
\end{align}
where the penultimate equality follows from Lemma~\ref{lemma:pb-localization}.
Adding and subtracting $\frac{3}{2}P_{n}(\ell_{\widehat{\rho}} -
\ell_{\gamma})$ we obtain the following deterministic identity:
\begin{align}
  &R(\widehat{\rho}) -
  \inf_{\gamma \in \probas (\Theta)}
  \left\{
    R(\gamma) + \frac{3}{4}V(\gamma)
    + \frac{3}{2\beta} \kl(\gamma, \pi_{-\frac{\beta}{3}R})
  \right\}
  \\
  &=
  \frac{3}{2}(P-P_{n})(\ell_{\widehat{\rho}} - \ell_{\gamma})
  - \frac{3}{2\beta} \kl(\widehat{\rho}, \pi_{-\frac{\beta}{3}R})
  - \frac{5}{4}V(\gamma)
  + \frac{1}{2}V(\widehat{\rho})
  \\
  &\quad\quad
  +\frac{3}{2}\left(
    P_{n}(\ell_{\widehat{\rho}} - \ell_{\gamma})
    - \frac{1}{\beta} \kl(\gamma, \pi)
    + \frac{1}{\beta} \kl(\widehat{\rho}, \pi)
  \right).
  \label{eq:excess-risk-decomposition}
\end{align}

Plugging in the result of
Lemma~\ref{lemma:q-aggregation-quadratic-loss-negative-term}
into \eqref{eq:excess-risk-decomposition}, and multiplying both sides by
$\frac{2}{3}$, we obtain
\begin{align}
  &
  \frac{2}{3}
  \left[
    R(\widehat{\rho}) -
    \inf_{\gamma \in \probas (\Theta)}
    \left\{
      R(\gamma) + \frac{3}{4}V(\gamma) + \frac{3}{2\beta} \kl(\gamma,
      \pi_{-\frac{\beta}{3}R})
    \right\}
  \right]
  \\
  &\leq
  (P-P_{n})(\ell_{\widehat{\rho}} - \ell_{\gamma})
  - \frac{1}{\beta} \kl(\widehat{\rho}, \pi_{-\frac{\beta}{3}R})
  - \frac{5}{6}V(\gamma)
  + \frac{1}{3}V(\widehat{\rho})
  -\frac{1}{2}\langle \widehat{\rho}, \|f - f_{\gamma}\|_{n}^{2} \rangle
  + \frac{1}{2}V_{n}(\gamma)
  \\
  &=
  (P-P_{n})(\ell_{\widehat{\rho}} - \ell_{\gamma})
  -\frac{1}{18}\langle \widehat{\rho}, \|f - f_{\gamma}\|_{n}^{2} \rangle
  - \frac{1}{18}\langle \widehat{\rho}, \|f - f_{\gamma}\|_{2}^{2}\rangle
  - \frac{1}{2\beta} \kl(\widehat{\rho}, \pi_{-\frac{\beta}{3}R})
    \label{eq:three-terms-T1}
  \\&\quad\quad +\frac{1}{3}V(\widehat{\rho})
  +\frac{1}{18}\langle \widehat{\rho}, \|f - f_{\gamma}\|_{2}^{2} \rangle
  - \frac{8}{18}\langle \widehat{\rho}, \|f - f_{\gamma}\|_{n}^{2} \rangle
  - \frac{1}{2\beta} \kl(\widehat{\rho}, \pi_{-\frac{\beta}{3}R})
    \label{eq:three-terms-T2}
  \\&\quad\quad + \frac{1}{2}V_{n}(\gamma) - \frac{5}{6}V(\gamma)
    \label{eq:three-terms-T3}.
\end{align}
It remains to obtain high-probability upper bounds on each of the three terms
\eqref{eq:three-terms-T1}, \eqref{eq:three-terms-T2} and
\eqref{eq:three-terms-T3} and use the union bound.
We show how to bound each term separately in the three lemmas stated below and
proved at the end of this section.

To bound the term \eqref{eq:three-terms-T1}, we combine Talagrand's contraction
lemma with the usual PAC-Bayes analysis based on the convex-conjugate duality
of the moment generating function and the $\kl$ divergence. This part of the analysis closely follows 
that of Lecué and Rigollet~\cite{lecue2014q-aggregation}.
\begin{lemma}
  \label{lemma:bounding-T1}
  Fix any $\gamma \in \probas (\Theta)$. Then, for any $\beta \in (0, n/(576b^{2})]$
  and for any $\delta \in (0,1)$, with probability at least $1-\delta/3$, it holds
  that
  \begin{align}
    &\sup_{\rho \in \probas (\Theta)}\left\{
      (P-P_{n})(\ell_{\rho} - \ell_{\gamma})
      -\frac{1}{18}\langle \rho, \|f - f_{\gamma}\|_{n}^{2} \rangle
      -\frac{1}{18}\langle \rho, \|f - f_{\gamma}\|_{2}^{2} \rangle
      -\frac{1}{2\beta}\kl(\rho, \pi_{-\frac{\beta}{3}R})
    \right\}
    \\
    &\lesssim
    \frac{\log(3/\delta)}{\beta}.
  \end{align}
\end{lemma}

We bound the term \eqref{eq:three-terms-T2} using the relation \eqref{eq:pacbayes} combined with Bernstein's concentration inequality. 
\begin{lemma}
  \label{lemma:bounding-T2}
  Fix any $\gamma \in \probas (\Theta)$. Then, for any $\beta \in (0, n/(50b^{2})]$
  and for any $\delta \in (0,1)$, with probability at least $1-\delta/3$, it holds
  that
  \begin{align}
    \sup_{\rho \in \probas (\Theta)}\left\{
      \frac{1}{3}V(\rho)
      + \frac{1}{18}\langle \rho, \|f - f_{\gamma}\|_{2}^{2} \rangle
      - \frac{8}{18}\langle \rho, \|f - f_{\gamma}\|_{n}^{2} \rangle
      - \frac{1}{2\beta}\kl(\rho, \pi_{-\frac{\beta}{3}R})
    \right\}
    \lesssim
    \frac{\log(3/\delta)}{\beta}.
  \end{align}
\end{lemma}

Finally, since $\gamma \in \probas (\Theta)$ is a fixed vector that does not depend
on the data, we bound \eqref{eq:three-terms-T3} via a single application of 
Bernstein's inequality.
\begin{lemma}
  \label{lemma:bounding-T3}
  Fix any $\gamma \in \probas (\Theta)$. Then,
  for any $\delta \in (0,1)$, with probability at least $1-\delta/3$ it holds
  that
  \begin{align}
    \frac{1}{2}V_{n}(\gamma) - \frac{5}{6}V(\gamma)
    \lesssim \frac{b^{2}\log(3/\delta)}{n}.
  \end{align}
\end{lemma}

The proof of Theorem~\ref{thm:pac-bayes-q-aggregation} follows by bounding the
terms \eqref{eq:three-terms-T1},~\eqref{eq:three-terms-T2} and
\eqref{eq:three-terms-T3} via
Lemma~\ref{lemma:bounding-T1},~\ref{lemma:bounding-T2} and
\ref{lemma:bounding-T3}, respectively, and applying the union bound.
\hfill\qed

\subsubsection{Proof of Lemma~\ref{lemma:bounding-T1}}
Define
\begin{align}
  \label{eq:T1-lemma-proof-U-definition}
   U
   &=
   2\beta
   \sup_{\rho \in \probas (\Theta)}\left\{
      (P-P_{n})(\ell_{\rho} - \ell_{\gamma})
      -\frac{1}{18}\langle \rho, \|f - f_{\gamma}\|_{n}^{2} \rangle
      -\frac{1}{18}\langle \rho, \|f - f_{\gamma}\|_{2}^{2} \rangle
      -\frac{1}{2\beta}\kl(\rho, \pi_{-\frac{\beta}{3}R})
   \right\}
   \\
   &=
   \sup_{\rho \in \probas (\Theta)}\left\{
      2\beta(P-P_{n})(\ell_{\rho} - \ell_{\gamma})
      -\frac{\beta}{9}\langle \rho, \|f - f_{\gamma}\|_{n}^{2} \rangle
      -\frac{\beta}{9}\langle \rho, \|f - f_{\gamma}\|_{2}^{2} \rangle
      -\kl(\rho, \pi_{-\frac{\beta}{3}R})
   \right\}.
\end{align}
Let $\sigma_{1},\dots,\sigma_{n}$ denote i.i.d.\ Rademacher signs.
And let $X_{i}',Y_{i}'$ be an independent copy of $X_{i},Y_{i}$.
Let
$$
  P_{n}^{\sigma}h = \frac{1}{n}\sum_{i=1}^{n}\sigma_{i}h(X_{i}).
$$
Throughout the remained of the proof we use $X$ and $Y$ as shorthand notations for the sequences $X_{1},\dots,X_{n}$ and $Y_{1},\dots,Y_{n}$, respectively.
Then, symmetrization and contraction gives us
\begin{align}
  &\E_{(X,Y)}\left[
    \exp\left(U\right)
  \right]
  \\
  &\leq
  \E_{(X,Y)} \E_{\sigma} \left[
    \exp\left(
     2\cdot
     \sup_{\rho \in \probas (\Theta)}\left\{
        2\beta P_{n}^{\sigma} (\ell_{\rho} - \ell_{\gamma})
        -\frac{\beta}{9}\langle \rho, \|f - f_{\gamma}\|_{n}^{2} \rangle
        -\frac{1}{2}\kl(\rho, \pi_{-\frac{\beta}{3}R})
     \right\}
   \right)
  \right]
  \\
  &\leq
  \E_{(X,Y)} \E_{\sigma} \left[
    \exp\left(
     \sup_{\rho \in \probas (\Theta)}\left\{
        4\beta P_{n}^{\sigma} (\ell_{\rho} - \ell_{\gamma})
        -\frac{2\beta}{9}\langle \rho, \|f - f_{\gamma}\|_{n}^{2} \rangle
        -\kl(\rho, \pi_{-\frac{\beta}{3}R})
     \right\}
   \right)
  \right]
  \\
  &\leq
  \E_{(X,Y)} \E_{\sigma} \left[
    \exp\left(
     \sup_{\rho \in \probas (\Theta)}\left\{
        16b\beta P_{n}^{\sigma} (f_{\rho} - f_{\gamma})
        -\frac{2\beta}{9}\langle \rho, \|f - f_{\gamma}\|_{n}^{2} \rangle
        -\kl(\rho, \pi_{-\frac{\beta}{3}R})
     \right\}
   \right)
  \right],
  \label{eq:objective-linear-in-rho}
\end{align}
where the last inequality follow via contraction principle as in
\cite[Eq.~(3.11)]{lecue2014q-aggregation}. In particular, the quadratic loss is Lipschitz over
bounded domain. That is, for any $x, y$:
\begin{align}
  &\left|
    \left(\ell_{\rho}(x,y) - \ell_{\gamma}(x,y)\right)
    - \left(\ell_{\rho'}(x,y) - \ell_{\gamma}(x,y)\right)
  \right|
  \\
  &=
  \left|
    (f_{\rho}(x) - f_{\rho'}(x))(f_{\rho}(x) + f_{\rho'}(x) - 2y)
  \right|
  \\
  &\leq
  4b
  \left|
    f_{\rho}(x) - f_{\rho'}(x)
  \right|
  \\
  &=
  4b
  \left|
    (f_{\rho}(x) - f_{\gamma}(x)) - (f_{\rho'}(x) - f_{\gamma}(x))
  \right|.
\end{align}

We are ready to apply the PAC-Bayesian machinery to
\eqref{eq:objective-linear-in-rho}. Using \eqref{eq:pacbayes}, we have 
\begin{align}
  &\E_{(X,Y)}\left[
    \exp\left(U\right)
  \right]
  \\
  &\leq
  \E_{(X,Y)} \E_{\sigma} \E_{J\sim\pi_{-\frac{\beta}{3}R}}
  \left[
    \exp\left(
      \frac{16b\beta}{n}\sum_{i=1}^{n}\sigma_{i}(f_{J}(X_{i}) - f_{\gamma}(X_{i}))
      - \frac{2\beta}{9}\|f_{J} - f_{\gamma}\|_{n}^{2}
   \right)
  \right]
  \\
  &=
  \E_{J\sim\pi_{-\frac{\beta}{3}R}} \E_{(X,Y)} \left[
    \exp\left(
        - \frac{2\beta}{9}\|f_{J} - f_{\gamma}\|_{n}^{2}
    \right)
    \E_{\sigma}
    \left[
      \exp\left(
        \frac{16b\beta}{n}\sum_{i=1}^{n}\sigma_{i}(f_{J}(X_{i}) - f_{\gamma}(X_{i}))
     \right)
    \right]
  \right].
  \label{eq:lemma-T1-step-before-hoeffding}
\end{align}
Now we will upper bound the expectation with respect to the Rademacher signs
via Hoeffding's lemma.
We have
\begin{align}
  \E_{\sigma}
  \left[
    \exp\left(
      \frac{16b\beta}{n}\sum_{i=1}^{n}\sigma_{i}(f_{J}(X_{i}) - f_{\gamma}(X_{i}))
   \right)
  \right]
  &=
  \prod_{i=1}^{n}
  \E_{\sigma_{i}}
  \left[
    \exp\left(
      \frac{16b\beta}{n}\sigma_{i}(f_{J}(X_{i}) - f_{\gamma}(X_{i}))
   \right)
  \right]
  \\
  &\leq
  \prod_{i=1}^{n}
  \exp\left(
    \frac{128 b^{2}\beta^{2}}{n^{2}}(f_{J}(X_{i}) - f_{\gamma}(X_{i})^{2})
  \right)
  \\
  &=
  \exp\left(
    \frac{128 b^{2}\beta^{2}}{n}\|f_{J} - f_{\gamma}\|_{n}^{2}
  \right).
\end{align}
Plugging the above into \eqref{eq:lemma-T1-step-before-hoeffding} yields
\begin{align}
  &\E_{(X,Y)}\left[
    \exp\left(U\right)
  \right]
  \leq
  \E_{J\sim\pi_{-\frac{\beta}{3}R}} \E_{(X,Y)} \left[
    \exp\left(
        - \frac{2\beta}{9}\|f_{J} - f_{\gamma}\|_{n}^{2}
    \right)
    \exp\left(
      \frac{128 b^{2}\beta^{2}}{n}\|f_{J} - f_{\gamma}\|_{n}^{2}
    \right)
  \right]
  \label{eq:T1-before-optimizing-lambda}.
\end{align}
For any $\beta \in (0, n/(576b^2)]$, the above is at most equal to one.
It follows that for any $\beta \in (0, n/(576b^2)]$ we have,
$$
  \E \exp(U) \leq 1
$$
and therefore, via Markov's inequality, with probability at least $1- \delta/3$,
it holds that
$$
  U \leq \log(3/\delta),
$$
which concludes the proof of this lemma.
\hfill\qed

\subsubsection{Proof of Lemma~\ref{lemma:bounding-T2}}
Define
\begin{align}
   U
   &=
   2\beta
   \sup_{\rho \in \probas (\Theta)}\left\{
     \frac{1}{3}V(\rho) + \frac{1}{18}
     \langle \rho, \|f - f_{\gamma}\|_{2}^{2} \rangle
      - \frac{8}{18}\langle \rho, \|f - f_{\gamma}\|_{n}^{2} \rangle
      - \frac{1}{2\beta}\kl(\rho, \pi_{-\frac{\beta}{3}R})
   \right\}
   \\
   &=
   \sup_{\rho \in \probas (\Theta)}\left\{
     \frac{2\beta}{3}V(\rho)
     + \frac{\beta}{9}
     \langle \rho, \|f - f_{\gamma}\|_{2}^{2} \rangle
      - \frac{8\beta}{9}\langle \widehat{\rho}, \|f - f_{\gamma}\|_{n}^{2} \rangle
      - \kl(\widehat{\rho}, \pi_{-\frac{\beta}{3}R})
   \right\}.
\end{align}
Next, observe that
$$
  V(\rho)
  =
  \langle \rho, \|f - f_{\gamma}\|_{2}^{2}\rangle
  - \|f_{\rho} - f_{\gamma}\|_{2}^{2}
  \leq
  \langle \rho, \|f - f_{\gamma}\|_{2}^{2}\rangle.
$$
Applying the above inequality, it follows that
$$
 U
 \leq
 \sup_{\rho \in \probas (\Theta)}\left\{
   \frac{7\beta}{9}
   \langle \rho, \|f - f_{\gamma}\|_{2}^{2} \rangle
    - \frac{8\beta}{9}\langle \rho, \|f - f_{\gamma}\|_{n}^{2} \rangle
    - \kl(\rho, \pi_{-\frac{\beta}{3}R})
 \right\}.
$$
Let $X$ be a shorthand for $X_{1},\dots,X_{n}$
and let $\sigma = (\sigma_{1},\dots,\sigma_{n})$ be a sequence of independent
Rademacher random signs.
Then, computing the above supremum explicitly and
applying the symmetrization inequality (cf.\ the proof of
Lemma~\ref{lemma:bounding-T1}) yields
\begin{align}
  &\E_{X}\left[
    \exp\left(U\right)
  \right]
  \\
  &=
  \E_{X} \E_{J \sim \pi_{-\frac{\beta}{3}R}}\left[
    \exp\left(
      \frac{\beta}{18}\left(
        14\|f_{J} - f_{\gamma}\|_{2}^{2}
        -16\|f_{J} - f_{\gamma}\|_{n}^{2}
      \right)
    \right)
  \right]
  \\
  &=
  \E_{J\sim\pi_{-\frac{\beta}{3}R}}
  \E_{X}\left[
    \exp\left(
      \frac{\beta}{18}\left(
        14\|f_{J} - f_{\gamma}\|_{2}^{2}
        -16\|f_{J} - f_{\gamma}\|_{n}^{2}
      \right)
    \right)
  \right]
  \\
  &=
  \E_{J\sim\pi_{-\frac{\beta}{3}R}}
  \E_{X}\left[
    \exp\left(
      \frac{\beta}{18}\left(
        15\|f_{J} - f_{\gamma}\|_{2}^{2}
        -15\|f_{J} - f_{\gamma}\|_{n}^{2}
        -\|f_{J} - f_{\gamma}\|_{2}^{2}
        -\|f_{J} - f_{\gamma}\|_{n}^{2}
      \right)
    \right)
  \right]
  \\
  &\leq
  \E_{J\sim\pi_{-\frac{\beta}{3}R}}
  \E_{X}
  \E_{\sigma}\left[
    \exp\left(
      \frac{\beta}{9}\left(
        \frac{15}{n}\sum_{i=1}^{n}\sigma_{i}(f_{J}(X_{i}) - f_{\gamma}(X_{i}))^{2}
        -\|f_{J} - f_{\gamma}\|_{n}^{2}
      \right)
    \right)
  \right]
  \\
  &=
  \E_{J\sim\pi_{-\frac{\beta}{3}R}}
  \E_{X}
    \exp\left(
      -\frac{\beta}{9}\|f_{J} - f_{\gamma}\|_{n}^{2}
    \right)
  \E_{\sigma}\left[
    \exp\left(
      \frac{15\beta}{9n}
        \sum_{i=1}^{n}\sigma_{i}(f_{J}(X_{i}) - f_{\gamma}(X_{i}))^{2}
    \right)
  \right].
\end{align}
We can now bound the expectation with respect to the Rademacher signs via
Hoeffding's inequality. In particular, we have
\begin{align}
  &\E_{X}\left[\exp(U)\right]
  \\
  &\leq
  \E_{J\sim\pi_{-\frac{\beta}{3}R}}
  \E_{X}
    \exp\left(
      -\frac{\beta}{9}\|f_{J} - f_{\gamma}\|_{n}^{2}
    \right)
  \prod_{i=1}^{n}
  \E_{\sigma_{i}}\left[
    \exp\left(
      \frac{15\beta}{9n}
      \sigma_{i}(f_{J}(X_{i}) - f_{\gamma}(X_{i}))^{2}
    \right)
  \right]
  \\
  &\leq
  \E_{J\sim\pi_{-\frac{\beta}{3}R}}
  \E_{X}
    \exp\left(
      -\frac{\beta}{9}\|f_{J} - f_{\gamma}\|_{n}^{2}
    \right)
    \prod_{i=1}^{n}
    \exp\left(
      \frac{225\beta^{2}}{162n^2}
      (f_{J}(X_{i}) - f_{\gamma}(X_{i}))^{4}
    \right)
  \\
  &\leq
  \E_{J\sim\pi_{-\frac{\beta}{3}R}}
  \E_{X}
    \exp\left(
      -\frac{\beta}{9}\|f_{J} - f_{\gamma}\|_{n}^{2}
    \right)
    \prod_{i=1}^{n}
    \exp\left(
      \frac{225\beta^{2}}{162n^2}
      \cdot 4b^{2}
      (f_{J}(X_{i}) - f_{\gamma}(X_{i}))^{2}
    \right)
  \\
  &=
  \E_{J\sim\pi_{-\frac{\beta}{3}R}}
  \E_{X}
    \exp\left(
      \left[\frac{900b^{2}\beta^{2}}{162n}
      - \frac{\beta}{9}\right]
      \|f_{J} - f_{\gamma}\|_{n}^{2}
    \right).
\end{align}
For any $\beta \in (0, n/(50b^{2})]$ the above is bounded by one.
Hence, by Markov's inequality,  for any $\delta \in (0,1)$ and
for any $\beta \in (0, n/(50b^{2})]$
we have
$$
  \mathbf{P}(U \geq \log(3/\delta)) \leq \frac{\delta}{3},
$$
which is what we wanted to show.
\hfill\qed

\subsubsection{Proof of Lemma~\ref{lemma:bounding-T3}}
Define the function $g$ by
$$
  g(X)
  =
  \langle \gamma, (f(X) - f_{\gamma}(X))^{2} \rangle
  =
  \E_{J\sim\gamma}[(f_{J}(X) - f_{\gamma}(X))^{2}].
$$
Then,
$$
  V_{n}(\gamma) = \frac{1}{n}\sum_{i=1}^{n}g(X_{i})
  \quad\text{and}\quad
  V(\gamma) = \E[g(X)].
$$
Note that
$
  |g(X)| \leq 4b^{2}
$,
and
\begin{align}
  \operatorname{Var}(g(X))
  &\leq
  \E g(X)^{2}
  = \E_{X} \left(\E_{J \sim \gamma}[(f_{J}(X) - f_{\gamma}(X))^{2}]\right)^{2}
  \\
  &\leq
  \E_{X} \E_{J \sim \gamma}[(f_{J}(X) - f_{\gamma}(X))^{4}]
  \leq
  4b^{2} \E_{X} \E_{J\sim\gamma}[(f_{J}(X) - f_{\gamma}(X))^{2}]
  \\
  &= 4b^{2} V(\gamma).
\end{align}
Hence, by Bernstein's inequality, with probability at least $1 - \delta/3$ we
have
\begin{align}
  V_{n}(\gamma) - V(\gamma)
  \leq
  \sqrt{\frac{2\cdot4b^{2} V(\gamma)\cdot \log(3/\delta)}{n}}
  +
  \frac{4b^{2}\log(3/\delta)}{3n}.
\end{align}
It follows that with probability at least $1-\delta/3$ we have
\begin{align}
  \frac{1}{2}V_{n}(\gamma) - \frac{5}{6}V(\gamma)
  &=
  \frac{1}{2}\left(V_{n}(\gamma) - V(\gamma)\right) - \frac{1}{3}V(\gamma)
  \\
  &\leq
  \frac{1}{2}\sqrt{\frac{2\cdot4b^{2} V(\gamma)\cdot \log(3/\delta)}{n}}
  - \frac{1}{3}V(\gamma)
  +
  \frac{1}{2}\frac{4b^{2}\log(3/\delta)}{3n}
  \\
  &\leq
  \sup_{x \in \R}\bigg\{
    \frac{1}{2}x\sqrt{\frac{2\cdot4b^{2}\cdot \log(3/\delta)}{n}}
    - \frac{1}{3}x^{2}
  \bigg\}
  + \frac{1}{2}\frac{4b^{2}\log(3/\delta)}{3n}
  \\
  &\lesssim
  \frac{b^{2}\log(3/\delta)}{n},
\end{align}
which is what we wanted to show.
\hfill\qed

\subsection{Proof of Theorem~\ref{thm:exchangablepriors}}
The proof follows the proof of Theorem \ref{thm:pac-bayes-q-aggregation} almost verbatim. Let us discuss the necessary changes. First, replacing $R(\cdot)$ with $R_n^{\prime}(\cdot)$, $V(\cdot)$ with $V_n^{\prime}(\cdot)$, $P$ with $P_n^{\prime}$, and $\|\cdot\|_2$ with $\|\cdot\|^{\prime}_{n}$ (the versions with primes correspond to empirical means with respect to the second half of the sample) in the first part of the proof of Theorem \ref{thm:pac-bayes-q-aggregation}, we obtain for any $\gamma \in \mathcal{P}(\Theta)$ that possibly depends on the entire sample of size $2n$,
We have
\begin{align}
  &R^{\prime}_n(\widehat{\rho}) -
 \left(
    R_n^{\prime}(\gamma) + \frac{3}{4}V_n^{\prime}(\gamma)
    + \frac{3}{2\beta} \kl(\gamma, \pi_{-\frac{\beta}{6}(R_n + R^{\prime}_n)})
  \right)
   \\
  &=P_n^{\prime}(\ell_{\widehat{\rho}} - \ell_{\gamma})
  - \frac{3}{4}V_n^{\prime}(\gamma)
    + \frac{3}{2\beta} \Big\{
    \kl(\widehat{\rho}, \pi_{-\frac{\beta}{6}(R_n + R^{\prime}_n)})  - \kl(\gamma, \pi_{-\frac{\beta}{6}(R_n + R^{\prime}_n)}) \Big\}
    \\
    &\qquad- \frac{3}{2\beta} \kl(\widehat{\rho}, \pi_{-\frac{\beta}{6}(R_n + R^{\prime}_n)})
  \\
  &=P_n^{\prime}(\ell_{\widehat{\rho}} - \ell_{\gamma})
  - \frac{3}{4}V_n^{\prime}(\gamma)
    + \frac{3}{2\beta} \Big\{ \kl(\widehat{\rho}, \pi) - \kl(\gamma, \pi) \Big\}
    \\
    &\qquad+ \frac{1}{4}\langle \widehat{\rho} - \gamma, R_{n}^{\prime} \rangle + \frac{1}{4}\langle \widehat{\rho} - \gamma, R_{n} \rangle
  - \frac{3}{2\beta} \kl(\widehat{\rho}, \pi_{-\frac{\beta}{6}(R_n + R^{\prime}_n)})
  \\
  &=
  \frac{5}{4}(P_n^{\prime} - P_n)(\ell_{\widehat{\rho}} - \ell_{\gamma})
  - V_n^{\prime}(\gamma)
  + \frac{1}{4}V_n^{\prime}(\widehat{\rho})
  - \frac{3}{2\beta} \kl(\gamma, \pi)
  + \frac{3}{2\beta} \kl(\widehat{\rho}, \pi)
  \\
  &\qquad+ \frac{3}{2}P_n(\ell_{\widehat{\rho}} - \ell_{\gamma}) + \frac{1}{4}V_n(\widehat{\rho}) - \frac{1}{4}V_n(\gamma)- \frac{3}{2\beta} \kl(\widehat{\rho}, \pi_{-\frac{\beta}{6}(R_n + R^{\prime}_n)}).
  \label{eq:lastline}
\end{align}
where the penultimate equality follows from Lemma~\ref{lemma:pb-localization}. By Lemma \ref{lemma:q-aggregation-quadratic-loss-negative-term} we have
\[
  R_{n}(\widehat{\rho}) - R_{n}(\gamma)
    - \frac{1}{\beta}\kl(\gamma, \pi)
    + \frac{1}{\beta}\kl(\widehat{\rho}, \pi)
    \leq
    -\frac{1}{2}\langle \widehat{\rho}, \|f - f_{\gamma}\|_{n}^{2} \rangle
    + \frac{1}{2}V_{n}(\gamma).
\]
Plugging this inequality into \eqref{eq:lastline}, we have
\begin{align*}
  &R^{\prime}_n(\widehat{\rho}) -
 \left(
    R_n^{\prime}(\gamma) + \frac{3}{4}V_n^{\prime}(\gamma)
    + \frac{3}{2\beta} \kl(\gamma, \pi_{-\frac{\beta}{6}(R_n + R^{\prime}_n)})
  \right)
   \\
  &\le
  \frac{5}{4}(P_n^{\prime} - P_n)(\ell_{\widehat{\rho}} - \ell_{\gamma})
  - V_n^{\prime}(\gamma)
  + \frac{1}{4}V_n^{\prime}(\widehat{\rho})
  - \frac{3}{2\beta} \kl(\widehat{\rho}, \pi_{-\frac{\beta}{6}(R_n + R^{\prime}_n)})
  \\
  &\quad +\frac{1}{2}V_n(\gamma) - \frac{3}{4}\langle \widehat{\rho}, \|f - f_{\gamma}\|_{n}^{2} \rangle + \frac{1}{4}V_n(\widehat{\rho}).
  \\
  &= \frac{5}{4}(P_n^{\prime} - P_n)(\ell_{\widehat{\rho}} - \ell_{\gamma}) - \frac{3}{16}\langle \widehat{\rho}, \|f - f_{\gamma}\|_{n}^{2} \rangle - \frac{3}{16}\langle \widehat{\rho}, (\|f - f_{\gamma}\|_{n}^{\prime})^{2} \rangle  - \frac{3}{4\beta} \kl(\widehat{\rho}, \pi_{-\frac{\beta}{6}(R_n + R^{\prime}_n)})
  \\
  &\quad+\frac{1}{4}V_n^{\prime}(\widehat{\rho})+\frac{1}{4}V_n(\widehat{\rho}) + \frac{3}{16}\langle \widehat{\rho}, (\|f - f_{\gamma}\|_{n}^{\prime})^{2}\rangle - \frac{9}{16}\langle \widehat{\rho}, \|f - f_{\gamma}\|_{n}^{2} \rangle - \frac{3}{4\beta} \kl(\widehat{\rho}, \pi_{-\frac{\beta}{6}(R_n + R^{\prime}_n)})
  \\
  &\quad+ \frac{1}{2}V_n(\gamma) - V_n^{\prime}(\gamma).
\end{align*}
As in the proof of Theorem \ref{thm:pac-bayes-q-aggregation}, we individually bound each term appearing in the last three lines. To do so, we repeat the lines of the proof of Lemma \ref{lemma:bounding-T1}, Lemma \ref{lemma:bounding-T2} and Lemma \ref{lemma:bounding-T3}. There are two observations that allow us to work with the prior distribution $\pi_{-\frac{\beta}{6}(R_n + R^{\prime}_n)}$ that depends on the entire sample of size $2n$. In fact, one can still use the symmetrization since the random measure $\pi_{-\frac{\beta}{6}(R_n + R^{\prime}_n)}$ does not change if we \say{swap} any two observations between the the first and the second half of the sample of size $2n$. Furthermore, after we apply the symmetrization step, all the randomness in our analysis comes from the Rademacher random signs. The remaining technical steps are exactly the same as in the proofs of Lemma \ref{lemma:bounding-T1}, and Lemma \ref{lemma:bounding-T2}. It is only left to modify the argument of Lemma \ref{lemma:bounding-T3}. The only difference is that $\gamma$ depends on the entire sample of size $2n$ and we cannot use the Bernstein inequality as before. Nevertheless, we can still use the exchangeability of our sample. For any $h > 0$ we have for sequence of independence Rademacher random variables (using the notation of the proof of Lemma \ref{lemma:bounding-T3})
\begin{align*}
&\E\exp\left(h n\left(\frac{1}{2}V_n(\gamma) - V_n^{\prime}(\gamma)\right)\right)
\\
&=\E\E_{\sigma}\exp\left(\frac{3h}{4}\sum\limits_{i = 1}^n\left(\sigma_i(g(X_i) - g(X_{n+i})) - \frac{1}{3}g(X_i) - \frac{1}{3}g(X_{n+i})\right)\right)
\\
&\le\E\E_{\sigma}\exp\left(\frac{3h}{2}\sum\limits_{i = 1}^n\left(\sigma_ig(X_i)-\frac{1}{3}g(X_i)\right)\right)
\\
&\le\E\exp\left(\frac{36h^2}{8}\sum\limits_{i = 1}^nb^2g(X_i)-\frac{1}{2}\sum\limits_{i = 1}^ng(X_i)\right) = 1,
\end{align*}
for $h = \frac{1}{3b}$. Here we used that $g(X) \ge 0$. By Markov's inequality, with probability at least $1 - \delta/3$,
\[
\frac{1}{2}V_n(\gamma) - V_n^{\prime}(\gamma) \le \frac{3b\log(3/\delta)}{n}.
\]
The claim follows as before.
\hfill\qed

\subsection{Proof of Theorem~\ref{thm:transductive-application}}

Before presenting the proof, we make some remarks. It appears that the existing analysis of aggregation procedures \cite{lecue2009aggregation, mendelson2017aggregation, mendelson2019unrestricted} goes through the complexity of the class formed by pairwise differences between functions in the reference class. That is, the existing upper bounds are written in term of the localized complexity of the class $\mathcal F - \mathcal F = \{f - g: f, g \in \mathcal F\}$. While for convex classes $\mathcal F - \mathcal F$ has essentially the same local complexity as $\mathcal F$ itself, the gap can be significant for non-convex classes as shown by the following example.

\begin{example}[The local complexity of the class $\mathcal F - \mathcal F$ can be larger than that of $\mathcal F$]
Consider the class of thresholds in $\mathbb{Z}$. That is, define $\mathcal F = \{x \mapsto \indic{x \ge t}: t \in \mathbb{Z}\}$. It is immediate to verify that the star number $s$ of $\mathcal F$ is equal to $2$. We further observe that the class $\mathcal F - \mathcal F$ contains the class of singletons $\mathcal G = \{x \mapsto \indic{x = t}: t \in \mathbb{Z}\}$. However, the star number of $\mathcal G$ is equal to infinity. Informally, while each function in $\mathcal F$ has at most two \say{neighbors} when restricted to a finite sample, the class $\mathcal F - \mathcal F$ is more expressive and each function has a rich neighborhood.
\end{example}
Thus, if one wants to incorporate the star number in the bound such as \eqref{eq:rakhlinsridharan}, we need to avoid the existing techniques that focus on the class $\mathcal F - \mathcal F$. Our techniques are particularly suitable for this purpose.

\begin{proof} The key point is to upper bound $\kl(\gamma, \pi_{-\frac{\beta}{6}(R_n + R^{\prime}_{n})})$ in Theorem \ref{thm:exchangablepriors}. Fix $\beta = c_1n$. Observe that our data dependent projection class $\mathcal F|_{S_{2n}} = \{(f(X_1), \ldots, f(X_{2n})): f \in \mathcal F\}$ is finite and can be seen as a subset of $\{0, 1\}^{2n}$. Moreover, its VC dimension is bounded by $d$ and its star number is bounded by $s$. Observe that $\beta(R_n(f) + R^{\prime}_{n}(f))$ is proportional to the number of mistakes that $f$ does on the sample of size $2n$. Assume that $\gamma$ puts its whole mass on any $f^{\prime} \in \mathcal F|_{S_{2n}}$ that minimizes $R_n(f) + R_n^{\prime}(f)$.
We pick one of these minimizers and denote it by $f^{\prime}$. We have
\begin{align}
&\kl(\gamma, \pi_{-\frac{\beta}{6}(R_n + R^{\prime}_{n})})
\\
&= \log\left(\sum\limits_{f \in \mathcal F|_{S_{2n}} }\exp(-c_1n(R_n(f) + R^{\prime}_{n}(f) - R_n(f^{\prime}) - R^{\prime}_n(f^{\prime}))/6)\right)
\\
&=\log\left(\sum\limits_{k = 0}^{2n}\big|\{f \in \mathcal F|_{S_{2n}}: n(R_n(f) + R^{\prime}_n(f) - R_n(f^{\prime}) - R^{\prime}_n(f^{\prime})) = k\}\big|\cdot\exp(-c_1k/6)\right)
\label{eq:localklvc}
\end{align}
We now want to bound
\[
|\{f \in \mathcal F|_{S_{2n}}: n(R_n(f) + R^{\prime}_n(f) - R_n(f^{\prime}) - R_n^{\prime}(f^{\prime})) = k\}| = (I)_k,
\]
for all $k = 0, \ldots, 2n$. Using the bound in \cite[Theorem 10]{hanneke2015minimax}, one can verify that
\begin{align*}
&\bigl|\{x \in S_{2n}: \textrm{there exist}\ f, g \in \mathcal F|_{S_{2n}}\ \textrm{such that}\ f(x) \neq g(x),
\\
&\qquad\qquad\qquad \textrm{and} \ R_n(f) + R^{\prime}_n(f) = R_n(g) + R^{\prime}_n(g) = k/n + R_n(f^{\prime}) + R^{\prime}_n(f^{\prime})\}\bigr|
\\
&\qquad\le (s + 1)(n(R_n(f^{\prime}) + R^{\prime}_n(f^{\prime})) + k).
\end{align*}
Thus, by the standard Sauer-Vapnik-Chervonenkis lemma we have
\[
(I)_k \le \left(\frac{e(s+1)(n(R_n(f^{\prime}) + R^{\prime}_n(f^{\prime})) + k)}{d}\right)^d.
\]
Plugging this into \eqref{eq:localklvc} we have for some $c_2, c_3,c_4,c_5 > 0$,
\begin{align*}
\kl(\gamma, \pi_{-\frac{\beta}{6}(R_n + R^{\prime}_{n})}) &\le \log\left(\sum\limits_{k = 0}^{2n} \left(\frac{e(s+1)(n(R_n(f^{\prime}) + R^{\prime}_n(f^{\prime})) + k)}{d}\right)^d\exp(-c_1k/6)\right)
\\
&\le  \log\Bigg(c_2\left(e(s+1)\right)^d \Bigg( \left(\frac{c_3n(R_n(f^{\prime}) + R^{\prime}_n(f^{\prime}) + d)}{d} \right)^d
\\
&\qquad+ \sum\limits_{k = c_4(R_n(f^{\prime}) + R^{\prime}_n(f^{\prime})) + d}^{\infty}\exp(c_5d \log(k/d) - c_1k/6)\Bigg)\Bigg).
\end{align*}
Using $\log(k/d) \le k/d$ for large enough $k$, it is easy to see that the last summand does not play any role and is absorbed by absolute constants.
Finally, we use the standard relative Vapnik-Chervonenkis bound \cite{vapnik1974theory} to prove that, with probability at least $1 - \delta$,
\[
R_n(f^{\prime}) + R^{\prime}_n(f^{\prime}) \le 4\inf\limits_{f \in \mathcal F}R(f) + \frac{c(d\log(en/d) + \log(1/\delta))}{n}.
\]
where $c > 0$ is some absolute constant. Combining the bounds and using the union bound, we obtain the bound 
\[
   \P\left(
      R^{\prime}_{n}(\widehat{\rho}) - \min\limits_{f \in \F}R^{\prime}_{n}(f)
      \geq c_6\left(\frac{d}{n}\log\left(\frac{es\left(nR(f^{\star}) + d\log(\frac{n}{d}) + \log(\frac{4}{\delta})\right)}{d}\right)
      + \frac{\log(\frac{4}{\delta})}{n}\right)
    \right)
    \leq \delta,
    \]
    where $c_6 > 0$ is some constant.
Using \cite[Lemma 20]{hanneke2016refined}, we simplify the expression under the logarithm and conclude the proof.
\end{proof}

\subsection{Proofs from Section~\ref{sec:gauss-priors}}
\label{sec:rema-proofs-gaussian}

We introduce additional notation.
Let $\norm{x}_{A} = \langle A x, x\rangle^{1/2}$ for $A$ a positive semi-definite matrix and $x \in \R^d$. In the random-design setting, %
letting $\theta^\star = \argmin_{\theta \in \R^d} R (f_\theta)$, we have $R (f_\theta) - R (f_{\theta^\star}) %
= \E [ (f_\theta (X) - f_{\theta^\star} (X))^2 ]
= \norm{\theta - \theta^\star}_\Sigma^2
$ for every $\theta \in \R^d$.
Likewise, in the fixed-design setting, $R (f_\theta) - R(f_{\theta^\star})
= \norm{\theta - \theta^\star}_{\wh \Sigma_n}^2
$
Recall, for $\lambda > 0$, the optimal regularized parameter is given by
\begin{equation*}
  \theta_\lambda
  = \argmin_{\theta \in \R^d} \big\{ R (\theta) + \lambda \norm{\theta}^2 \big\}
  = \argmin_{\theta \in \R^d} \big\{ \norm{\theta - \theta^\star}_\Sigma^2 + \lambda \norm{\theta}^2 \big\}
  = (\Sigma + \lambda I_d)^{-1} \Sigma \theta^\star
  ,
\end{equation*}
such that %
\[
  \inf_{\theta \in \R^d} \big\{ R (\theta) + \lambda \norm{\theta}^2 \big\} - R (\theta^\star)
  = R (\theta_\lambda) - R (\theta^\star) + \lambda \norm{\theta_\lambda}^2
    = \lambda \innerp{(\Sigma + \lambda I_d)^{-1} \Sigma \theta^\star}{\theta^\star},
    \]
    and
    \[
  R (\theta_\lambda) - R (\theta^\star)
  = \lambda^2 \innerp{(\Sigma + \lambda I_d)^{-2} \Sigma \theta^\star}{\theta^\star}
    .
\]
We use these expressions in what follows.

\subsubsection{Proof of Proposition~\ref{prop:local-global-gaussian}}
\label{sec:proof-gaussian-complexities}

Letting $\lambda = \gamma/\beta$, we have for any $\theta \in \R^d$,
\begin{align*}
  &\beta R (f_{\theta}) + \gamma \norm{\theta}^2
    = \beta \big\{ R (\theta) + \lambda \norm{\theta}^2 \big\} %
  = \beta \big\{ \innerp{(\Sigma + \lambda I_d) (\theta - \theta_\lambda)}{\theta - \theta_\lambda} + R (\theta_\lambda) + \lambda \norm{\theta_\lambda}^2 \big\}
\end{align*}
where the last equality stems from the fact that the two functions of $\theta$ are quadratic, with the same Hessian, minimizer and minimum.
As a result, denoting by $\pi_\beta$ the distribution $\gaussdist (\theta_\lambda, \beta^{-1} (\Sigma + \lambda I_d)^{-1}) = \gaussdist (\theta_\lambda, (\beta \Sigma + \gamma I_d)^{-1})$, we have
\begin{align*}
  e^{-\beta R (\theta)/2} \pi (\di \theta)
  &= \frac{e^{- (\beta R (f_{\theta}) + \gamma \norm{\theta}^2)/2} \di \theta}{({2\pi} \gamma^{-1})^{d/2}} \\
  &= e^{-\beta \{ R (\theta_\lambda) + \lambda \norm{\theta_\lambda}^2 \}/2} \frac{(2\pi)^{d/2} \det [ (\beta \Sigma + \gamma I_d)^{-1} ]^{1/2}}{({2\pi} \gamma^{-1})^{d/2}} \pi_\beta (\di \theta) \\
  &= e^{-\beta \{ R (\theta_\lambda) + \lambda \norm{\theta_\lambda}^2 \}/2} \det \big( I_d + {\lambda^{-1}}\Sigma \big)^{-1/2} \pi_\beta (\di \theta)
    .
\end{align*}
This implies that the global complexity writes
\begin{align*}
  - \frac{2}{\beta} \log \bigg( \int_{\R^d} e^{- \beta R (\theta)/2} \pi (\di \theta) \bigg)
  &= - \frac{2}{\beta} \log \bigg( e^{-\beta \{ R (\theta_\lambda) + \lambda \norm{\theta_\lambda}^2 \}/2} \det \big( I_d + {\lambda^{-1}}\Sigma \big)^{-1/2} \bigg) \\
  &= R (\theta_\lambda) + \lambda \norm{\theta_\lambda}^2 + \beta^{-1} \log \det (I_d + \lambda^{-1} \Sigma)
    .
\end{align*}
In addition, the local complexity writes
\begin{align*}
  \frac{\int_{\R^d} R (\theta) e^{- \beta R (\theta)/2} \pi (\di \theta)}{\int_{\R^d} e^{- \beta R (\theta)/2} \pi (\di \theta)}
  &= \int_{\R^d} R (\theta) \pi_\beta (\di \theta) \\
  &= R (\theta^\star) + \int_{\R^d} \langle \Sigma (\theta - \theta^\star), \theta - \theta^\star\rangle \pi_\beta (\di \theta) \\
  &= R (\theta^\star) + \norm{\theta_\lambda - \theta^\star}_\Sigma^2 + \tr \bigg\{ \Sigma \int_{\R^d} (\theta - \theta_\lambda) (\theta - \theta_\lambda)^{\top} \pi_{\beta} (\di \theta) \bigg\} \\
  &= R (\theta_\lambda) + \beta^{-1} \, \tr [ (\Sigma + \lambda I_d)^{-1} \Sigma ]
    .
\end{align*}

We conclude with the proof of the two inequalities of Proposition~\ref{prop:local-global-gaussian}.
Let $\lambda_1 \geq \dots \geq \lambda_d \geq 0$ denote the ordered eigenvalues of $\Sigma$, counted with multiplicity, so that $\lambda_1 = \opnorm{\Sigma}$ and
\begin{align*}
  \tr \big[ (\Sigma + \lambda I_d)^{-1} \Sigma \big]
  &= \sum_{i=1}^d \frac{\lambda_i}{\lambda_i + \lambda}
    = \sum_{i=1}^d u_i , \\
  \log \det \big( I_d + \lambda^{-1} \Sigma \big)
  &= \sum_{i=1}^d \log \big( 1 + \lambda_i/\lambda \big)
    = \sum_{i=1}^d - \log (1 - u_i),
\end{align*}
where $u_i = \lambda_i / (\lambda_i + \lambda) \in [0, 1)$.
The bound $\tr [ (\Sigma + \lambda I_d)^{-1} \Sigma ] \leq \log \det (I_d + \lambda^{-1} \Sigma)$ then follows from the inequality $u_i \leq - \log (1 - u_i)$; the reverse inequality follows from the fact that, by convexity of $u \mapsto - \log (1-u)$ and since $u_i \leq u_1 \leq 1/2$ (by the assumption $\lambda \leq \lambda_1$)
\begin{equation*}
  \frac{- \log (1-u_i)}{u_i}
  \leq \frac{- \log (1 - u_1)}{u_1}
  \leq - 2 \log (1 - u_1)
  = 2 \log \bigg( 1 + \frac{\lambda_1}{\lambda} \bigg)
  .
\end{equation*}
The claim follows. \qed

\subsubsection{Proof of Proposition~\ref{prop:q-aggregation-gaussian}}
\label{sec:proof-gaussian-q-agg}

By the representer theorem (Proposition~\ref{prop:representer}), there exist $a_1, \dots, a_n, b_1, \dots, b_n, c_0 \in \R$ such that
\begin{equation*}
  \frac{\di \rho_n}{\di \theta} (\theta)
  = \frac{\di \pi}{\di \theta} (\theta) \frac{\di \rho_n}{\di \pi} (\theta)
  = (2\pi \gamma^{-1})^{-d/2} \exp \bigg\{ - \frac{\gamma \norm{\theta}^2}{2} + \sum_{i=1}^n \big( a_i \langle \theta, X_i\rangle^2 + b_i \innerp{\theta}{X_i} \big) + c_0 \bigg\}
\end{equation*}
so that $\wh \rho_n$ is a Gaussian distribution on $\R^d$.
We write $\wh \rho_n = \gaussdist (\mu_n, \Gamma_n)$ for $\mu_n \in \R^d$ and $\Gamma_n \in \R^{d\times d}$ a positive symmetric matrix.

Now, for a Gaussian distribution $\rho = \gaussdist (\mu, \Gamma)$, we have $f_\rho = f_\mu$ (by linearity of $f_\theta$ in $\theta$ and since $\int_{\R^d} \theta \rho (\di \theta) = \mu$), so that
\begin{equation*}
  R_n (\rho)
  = R_n (\mu)
  .
\end{equation*}
In addition, since $R_n$ is quadratic, we have
\begin{equation*}
  \langle \rho, R_n\rangle
  = \int_{\R^d} \big\{ R_n (\mu) + \innerp{\nabla R_n (\mu)}{\theta - \mu} + \innerp{\wh \Sigma_n (\theta - \mu)}{\theta - \mu} \big\} \rho (\di \theta)
  = R_n (\mu) + \tr (\wh \Sigma_n \Gamma )
  .
\end{equation*}
Finally, the Kullback-Leibler divergence writes:
\begin{equation*}
  \kll{\rho}{\pi}
  = \frac{1}{2} \big\{ - \log \det (\gamma \Gamma)
  + \tr ( \gamma \Gamma) - d + \gamma \norm{\mu}^2 \big\}
  .
\end{equation*}
Since $\wh \rho_n$ minimizes the $Q$-aggregation functional (with inverse temperature $\beta/2$) over Gaussian measures, by the previous computations and the definition~\eqref{eq:q-aggregation-estimator-definition} we have
\begin{align*}
  (\mu_n, \Gamma_n)
  &= \argmin_{\mu, \Gamma} \Big\{ (1-\alpha) \big[ R_n (\mu) + \tr ( \wh \Sigma_n \Gamma ) \big] + \alpha R_n (\mu) + \\
  &\qquad + \frac{1}{2\cdot \beta/2} \big\{ - \log \det (\gamma \Gamma) + \tr ( \gamma \Gamma) - d + \gamma \norm{\mu}^2 \big\}  \Big\} \\
  &= \argmin_{\mu, \Gamma} \Big\{ R_n (\mu) + (1-\alpha) \tr ( \wh \Sigma_n \Gamma) + \beta^{-1} \big\{ - \log \det (\gamma \Gamma) + \tr ( \gamma \Gamma) + \gamma \norm{\mu}^2 \big\}  \bigg\}
    .
\end{align*}
Since the terms in $\mu$ and $\Gamma$ are decoupled, we have, recalling that $\lambda = \gamma/\beta$,
\begin{align}
  \label{eq:q-aggregation-mean}
  \mu_n
  &= \argmin_{\mu \in \R^d} \big\{ R_n (\mu) + \lambda \norm{\mu}^2 \big\}
    = \wh \theta_\lambda  \\
  \label{eq:q-aggregation-covariance}
  \Gamma_n
  &= \argmin_{\Gamma} \beta^{-1} \big\{ - \log \det (\gamma \Gamma) + \tr \big[ \big( (1-\alpha) \beta \wh \Sigma_n + \gamma I_d \big) \Gamma \big] \big\} \nonumber \\
  &= \big( (1-\alpha) \beta \wh \Sigma_n + \gamma I_d \big)^{-1} \cdot \argmin_{\Gamma'} \big\{ -\log \det \Gamma' + \tr (\Gamma') - d \big\} \nonumber \\
  &= \big( (1-\alpha) \beta \wh \Sigma_n + \gamma I_d \big)^{-1}
\end{align}
where the penultimate inequality is obtained by making the change of variable $\Gamma = \big( (1-\alpha) \beta \wh \Sigma_n + \gamma I_d \big)^{-1} \Gamma'$ (and ignoring the $\log \det$ term that appears that does not depend on $\Gamma'$), while the last equation follows from the fact that $\tr (\Gamma') - d - \log \det (\Gamma') \geq 0$ for any symmetric positive matrix, with equality if and only if $\Gamma' = I_d$ (by applying the corresponding scalar inequality to each of the eigenvalues of $\Gamma'$).

In particular, since $f_{\wh \rho_n} = f_{\mu_n}$, the identity~\eqref{eq:q-aggregation-mean} implies that $f_{\wh \rho_n} = f_{\wh \theta_n}$ corresponds to ridge regression (irrespective of $\alpha \in [0, 1]$). The claim follows. \qed

\subsubsection{Proof of Theorem~\ref{thm:improper-ridge}}
\label{sec:proof-truncated}

Let $(X_1, Y_1), \dots, (X_{n+1}, Y_{n+1})$ be $n+1$ \iid samples from $P$.
For any regression procedure $f_n : (\R^d \times \R)^n \times \R^d \to \R$ and $1 \leq i \leq n+1$, denote $\wh f_n^{(i)} (x) = f_n \big( (X_j,Y_j)_{1 \leq j \leq n+1, j\neq i}, x \big)$; in particular, let $\wh f_n = \wh f_n^{(n+1)}$.
By exchangeability of the distribution of the sample $(X_1, Y_1), \dots, (X_{n+1},Y_{n+1})$, we have
\begin{align*}
  \E [ R (\wh f_n) ]
  &= \E [ (\wh f_n^{(n+1)} (X_{n+1}) - Y_{n+1})^2 ]
    = \E \bigg[ \frac{1}{n+1} \sum_{i=1}^{n+1} (\wh f_{n}^{(i)} (X_i) - Y_i)^2 \bigg]
    .
\end{align*}

In addition, define $ R_{n+1} (f) = (n+1)^{-1} \sum_{i=1}^{n+1} (f (X_i) - Y_i)^2$ for $f : \R^d \to \R$, and $R_{n+1} (\theta) =  R_{n+1} (f_\theta)$ for $\theta \in \R^d$.
Letting $\wh \theta_{\lambda,n+1} = \argmin_{\theta \in \R^d} \{  R_{n+1} (\theta) + \lambda \norm{\theta}^2 \}$, we have
\begin{align*}
  \inf_{\theta \in \R^d} \big\{ R (\theta) + \lambda \norm{\theta}^2 \big\}
  &= R (\theta_\lambda) + \lambda \norm{\theta_\lambda}^2 \\
  &= \E \big[  R_{n+1} (\theta_\lambda) + \lambda \norm{\theta_\lambda}^2 \big] \\
  &\geq \E \big[  R_{n+1} (\wh \theta_{\lambda, n+1}) + \lambda \norm{\wh \theta_{\lambda, n+1}}^2 \big] \\
  &\geq \E \big[  R_{n+1} (\wh \theta_{\lambda, n+1}) \big]
    .
\end{align*}
Combining the previous two inequalities gives
\begin{equation}
  \label{eq:proof-loo-bound}
  \E [ R (\wh f_n) ] - \inf_{\theta \in \R^d} \big\{ R (\theta) + \lambda \norm{\theta}^2 \big\}
  \leq \E \bigg[ \frac{1}{n+1} \sum_{i=1}^{n+1} \big\{ (\wh f_{n}^{(i)} (X_i) - Y_i)^2 - (\innerp{\wh \theta_{\lambda, n+1}}{X_i} - Y_i)^2 \big\} \bigg]
  .
\end{equation}
Now, for $i = 1, \dots, n+1$, denoting 
\[
h_i = \left\langle \left(\sum\nolimits_{j=1}^{n+1} X_j X_j^\top + \lambda (n+1) I_d\right)^{-1} X_i, X_i\right\rangle
\]
and $\theta_i = \wh \theta_{\lambda'}^{(i)}$, we have by Lemma~\ref{lem:update-leverage} below (with $A = \sum_{j\neq i} X_j X_j^\top + \lambda (n+1) I_d$, $\theta = \theta_i$ and $(x,y) = (X_i,Y_i)$) that
\begin{equation}
  \label{eq:proof-updated-loss}
  (\innerp{\wh \theta_{\lambda, n+1}}{X_i} - Y_i)^2
  = (1 - h_i)^2 (\innerp{\theta_i}{X_i} - Y_i)^2
  .
\end{equation}
We now separately analyze the estimators $\wh f_\lambda, \wh g_{\lambda, b}$ and $\wh g_{\lambda}$.

We start with $\wh f_\lambda$ which satisfies $\wh f_\lambda^{(i)} (X_i) = (1 - h_i)^2 \langle \theta_i, X_i\rangle$.
One therefore has, by~\eqref{eq:proof-updated-loss} and using that $0 \leq h_i \leq 1$, %
\begin{align*}
  & (\wh f_{n}^{(i)} (X_i) - Y_i)^2 - (\innerp{\wh \theta_{\lambda, n+1}}{X_i} - Y_i)^2 \\
  &= \big[ (1- h_i)^2 \innerp{\theta_i}{X_i} - Y_i \big]^2 - (1-h_i)^2 \big[ \innerp{\theta_i}{X_i} - Y_i \big]^2 \\
  &= - (1-h_i)^2 \big\{ 1 - (1-h_i)^2 \big\} \innerp{\theta_i}{X_i}^2 + \big\{ 1 - (1 - h_i)^2 \big\} Y_i^2 \\
  &\leq \big\{ 2 h_i - h_i^2 \big\} Y_i^2
    \leq 2 b^2 h_i
    .
\end{align*}
Plugging this inequality into~\eqref{eq:proof-loo-bound} gives:
\begin{equation*}
  \E [ R (\wh f_\lambda) ] - \inf_{\theta \in \R^d} \big\{ R (\theta) + \lambda \norm{\theta}^2 \big\}
  \leq \frac{2 b^2}{n+1} \, \E \bigg[ \sum_{i=1}^{n+1} h_i \bigg]
  .
\end{equation*}
In addition, denoting $\wh \Sigma_{n+1} = (n+1)^{-1} \sum_{i=1}^{n+1} X_i X_i^\top$,
\begin{align}
  \E \bigg[ \sum_{i=1}^{n+1} h_i \bigg]
  &= \E \bigg[ \sum_{i=1}^{n+1} \bigg\langle \bigg( \sum_{j=1}^{n+1} X_j X_j^\top + \lambda (n+1) I_d \bigg)^{-1} X_i, X_i \bigg\rangle \bigg] \nonumber \\
  &= \E \, \tr \bigg\{ \bigg( \sum_{j=1}^{n+1} X_j X_j^\top + \lambda (n+1) I_d \bigg)^{-1} \sum_{i=1}^{n+1} X_i X_i^\top \bigg\} \nonumber \\
  &= \E \,\tr \big\{ \big( \wh \Sigma_{n+1} + \lambda I_d \big)^{-1} \wh \Sigma_{n+1} \big\} \nonumber \\
  &\leq \tr \big[ (\Sigma + \lambda I_d)^{-1} \Sigma \big], \label{eq:proof-trace-concave}
\end{align}
where the last step is by
  convexity of $A \mapsto \tr (A^{-1})$ over positive semi-definite matrices \cite[Lemma~2.7]{carlen2010trace}, together with $\E [\wh \Sigma_{n+1}] = \Sigma$.
  This implies the bound~\eqref{eq:ridge-fw}.

  We now consider the truncated ridge estimator $\wh g_\lambda$, such that $\wh g_\lambda^{(i)} (X_i) = \psi_b (\langle \theta_i, X_i\rangle) = U_i \in [-b, b]$.
  Note that since $-b \leq Y_i \leq b$, we have $(Y_i - U_i)^2 %
  \leq (Y_i - \innerp{\theta_i}{X_i})^2$ and thus
  \begin{align*}
    (\wh g_{\lambda, b}^{(i)} (X_i) - Y_i)^2 - (\innerp{\wh \theta_{\lambda, n+1}}{X_i} - Y_i)^2
    &= (U_i - Y_i)^2 - (1-h_i)^2 \big( \innerp{\theta_i}{X_i} - Y_i \big)^2 \\
    &\leq (U_i - Y_i)^2 - \big[ (1-h_i) (U_i - Y_i) \big]^2 \\
    &= h_i (2 - h_i) (U_i - Y_i)^2 \\
    &\leq 8 b^2 h_i~,
  \end{align*}
  as $|U_i - Y_i| \leq 2 b$.
  Plugging this into the bound~\eqref{eq:proof-loo-bound} together with~\eqref{eq:proof-trace-concave} establishes~\eqref{eq:truncated-ridge-bound}.

  Finally, we conclude with the adaptive truncated ridge estimator $\wh g_\lambda$.
  Let $b_i = \max_{1 \leq j \leq n+1, \, j \neq i} |Y_j|$, so that $\wh g^{(i)}_\lambda (X_i) = \psi_{b_i} (\innerp{\theta_i}{X_i})$.
  Let $E_i$ denote the event $\{ |Y_i| > b_i \}$, such that $\P (E_i) \leq 1/(n+1)$ by exchangeability.
  Under $E_i^c$, we have $(Y_i - \psi_{b_i} (\innerp{\theta_i}{X_i}) )^2 \leq (Y_i - \innerp{\theta_i}{X_i})^2$, and the same argument as before shows that
  \begin{equation*}
    (\wh g_{\lambda, b}^{(i)} (X_i) - Y_i)^2 - (\innerp{\wh \theta_{\lambda, n+1}}{X_i} - Y_i)^2
    \leq 8 b_i^2 h_i \leq 8 b^2 h_i^2
    .
  \end{equation*}
  Under $E_i$, one can always bound $(\wh g_{\lambda, b}^{(i)} (X_i) - Y_i)^2 \leq 4 b^2$.
  The constant can be improved as follows: if $\innerp{\theta_i}{X_i}$ and $Y_i$ are of opposite signs, then again $(Y_i - \psi_{b_i} (\innerp{\theta_i}{X_i}) )^2 \leq (Y_i - \innerp{\theta_i}{X_i})^2$ so the bound $8 b^2 h_i^2$ still holds; if they are of the same sign, then $(Y_i - \psi_{b_i} (\innerp{\theta_i}{X_i}) )^2 \leq b^2$.
  As a result, we have in all cases:
  \begin{equation*}
    (\wh g_{\lambda, b}^{(i)} (X_i) - Y_i)^2 - (\innerp{\wh \theta_{\lambda, n+1}}{X_i} - Y_i)^2
    \leq 8 b^2 h_i^2 + b^2 \bm 1_{E_i}
    .
  \end{equation*}
Plugging this into the bound~\eqref{eq:proof-loo-bound}, together with the inequality~\eqref{eq:proof-trace-concave} and the bound $\P (E_i) \leq 1/(n+1)$ gives the claimed guarantee~\eqref{eq:adaptive-truncated-bound}.

\begin{lemma}
  \label{lem:update-leverage}
  Let $A$ be a symmetric positive $d \times d$ matrix, $b \in \R^d$ and $\theta = A^{-1} b$.
  Let $(x, y) \in \R^d \times \R$, and define $b' = b+ y x$, $A'  = A + x x^\top$ and $\theta' = A'^{-1} b'$.
  Define $h = \langle A'^{-1} x, x\rangle$.
  Then, we have
  \begin{equation}
    \label{eq:update-leverage}
    \innerp{\theta'}{x} - y
    = (1 - h) \big( \innerp{\theta}{x} - y \big)
    .
  \end{equation}
\end{lemma}

\begin{proof}
  It follows from the Sherman-Morrison identity that
  \begin{equation}
    \label{eq:sherman-leverage}
    A'^{-1} x
    = \Big( A^{-1} - \frac{A^{-1} x x^\top A^{-1}}{1 + \langle A^{-1} x, x\rangle} \Big) x
    = \frac{A^{-1} x}{1 + \langle A^{-1} x, x\rangle},
  \end{equation}
  so that, in particular, $h = \innerp{A'^{-1} x}{x} = \langle A^{-1} x, x\rangle / (1 + \langle A^{-1}x, x\rangle)$ (or alternatively $\innerp{A^{-1} x}{x} = h/(1-h)$), and thus~\eqref{eq:sherman-leverage} writes as $A'^{-1} x = (1 - h) A^{-1} x$.
  This implies that
  \begin{align*}
    \langle \theta', x\rangle - y
    &= \innerp{A'^{-1} b'}{x} - y \\
    &= \innerp{A'^{-1} x}{b + y x} - y \\
    &= (1-h) \innerp{A^{-1} x}{b + y x} - y \\
    &= (1-h) \innerp{A^{-1} b}{x} + (1 - h) \innerp{A^{-1} x}{x} y - y \\
    &= (1-h) (\innerp{\theta}{x} - y)
      .
  \end{align*}
The claim follows.
\end{proof}


\begin{thebibliography}{}

\end{thebibliography}


\begin{thebibliography}{10}

\bibitem{alquier2008pac}
P.~Alquier.
\newblock {PAC-Bayesian} bounds for randomized empirical risk minimizers.
\newblock {\em Mathematical Methods of Statistics}, 17:279--304, 2008.

\bibitem{alquier2021user}
P.~Alquier.
\newblock User-friendly introduction to {PAC}-{B}ayes bounds.
\newblock {\em arXiv preprint arXiv:2110.11216}, 2021.

\bibitem{anthony2009neural}
M.~Anthony and P.~L. Bartlett.
\newblock {\em Neural Network Learning: Theoretical Foundations}.
\newblock Cambridge University Press, 2009.

\bibitem{aronszajn1950theory}
N.~Aronszajn.
\newblock Theory of reproducing kernels.
\newblock {\em Transactions of the American Mathematical Society},
  68(3):337--404, 1950.

\bibitem{audibert2004aggregated}
J.-Y. Audibert.
\newblock Aggregated estimators and empirical complexity for least square
  regression.
\newblock {\em {Annales de l'Institut Henri Poincar{\'e}}, Probabilit{\'e}s et
  Statistiques}, 40(6):685--736, 2004.

\bibitem{audibert2004phd}
J.-Y. Audibert.
\newblock {\em {PAC-Bayesian} statistical learning theory}.
\newblock PhD thesis, Universit{\'e} Paris VI, 2004.

\bibitem{audibert2008deviation}
J.-Y. Audibert.
\newblock Progressive mixture rules are deviation suboptimal.
\newblock In {\em Advances in Neural Information Processing Systems 20}, pages
  41--48, 2008.

\bibitem{audibert2009fastrates}
J.-Y. Audibert.
\newblock Fast learning rates in statistical inference through aggregation.
\newblock {\em The Annals of Statistics}, 37(4):1591--1646, 2009.

\bibitem{audibert2010linear}
J.-Y. Audibert and O.~Catoni.
\newblock Linear regression through {PAC-Bayesian} truncation.
\newblock {\em Preprint arXiv:1010.0072}, 2010.

\bibitem{audibert2011robust}
J.-Y. Audibert and O.~Catoni.
\newblock Robust linear least squares regression.
\newblock {\em The Annals of Statistics}, 39(5):2766--2794, 2011.

\bibitem{azoury2001relative}
K.~S. Azoury and M.~K. Warmuth.
\newblock Relative loss bounds for on-line density estimation with the
  exponential family of distributions.
\newblock {\em Machine Learning}, 43(3):211--246, 2001.

\bibitem{barron1987bayes}
A.~R. Barron.
\newblock Are {Bayes} rules consistent in information?
\newblock In {\em Open Problems in Communication and Computation}, pages
  85--91. Springer, 1987.

\bibitem{bellec2018optimal}
P.~C. Bellec.
\newblock Optimal bounds for aggregation of affine estimators.
\newblock {\em The Annals of Statistics}, 46(1):30--59, 2018.

\bibitem{blanchard2018optlip}
G.~Blanchard and N.~M{\"{u}}cke.
\newblock Optimal rates for regularization of statistical inverse learning
  problems.
\newblock {\em Foundations of Computational Mathematics}, 18(4):971--1013,
  2018.

\bibitem{carlen2010trace}
E.~Carlen.
\newblock Trace inequalities and quantum entropy: an introductory course.
\newblock {\em Entropy and the quantum}, 529:73--140, 2010.

\bibitem{catoni2004statistical}
O.~Catoni.
\newblock {\em Statistical Learning Theory and Stochastic Optimization: Ecole
  d'Et{\'e} de Probabilit{\'e}s de Saint-Flour XXXI - 2001}, volume 1851 of
  {\em Lecture Notes in Mathematics}.
\newblock Springer-Verlag, 2004.

\bibitem{catoni2007pacbayes}
O.~Catoni.
\newblock {\em {PAC-Bayesian Supervised Classification: The Thermodynamics of
  Statistical Learning}}, volume~56 of {\em IMS Lecture Notes Monograph
  Series}.
\newblock Institute of Mathematical Statistics, 2007.

\bibitem{cesabianchi2004onlinebatch}
N.~Cesa-Bianchi, A.~Conconi, and C.~Gentile.
\newblock On the generalization ability of on-line learning algorithms.
\newblock {\em IEEE Transactions on Information Theory}, 50(9):2050--2057,
  2004.

\bibitem{cesabianchi2006plg}
N.~Cesa-Bianchi and G.~Lugosi.
\newblock {\em Prediction, Learning, and Games}.
\newblock Cambridge University Press, Cambridge, New York, USA, 2006.

\bibitem{cover2006elements}
T.~M. Cover and J.~A. Thomas.
\newblock {\em Elements of Information Theory}.
\newblock Wiley Series in Telecommunications and Signal Processing.
  Wiley-Interscience, New York, USA, 2nd edition, 2006.

\bibitem{dai2012q-aggregation}
D.~Dai, P.~Rigollet, and T.~Zhang.
\newblock Deviation optimal learning using greedy {$Q$}-aggregation.
\newblock {\em The Annals of Statistics}, 40(3):1878--1905, 2012.

\bibitem{dalalyan2008aggregation}
A.~Dalalyan and A.~B. Tsybakov.
\newblock Aggregation by exponential weighting, sharp {PAC-Bayesian} bounds and
  sparsity.
\newblock {\em Machine Learning}, 72(1):39--61, 2008.

\bibitem{dalalyan2022simple}
A.~S. Dalalyan.
\newblock Simple proof of the risk bound for denoising by exponential weights
  for asymmetric noise distributions.
\newblock {\em arXiv preprint arXiv:2212.12950}, 2022.

\bibitem{dalalyan2012sharp}
A.~S. Dalalyan and J.~Salmon.
\newblock Sharp oracle inequalities for aggregation of affine estimators.
\newblock {\em The Annals of Statistics}, 40(4):2327--2355, 2012.

\bibitem{dziugaite2017nonvacuous}
G.~K. Dziugaite and D.~M. Roy.
\newblock Computing nonvacuous generalization bounds for deep (stochastic)
  neural networks with many more parameters than training data.
\newblock In {\em Proceedings of {U}ncertainty in {A}rtificial {I}ntelligence},
  2017.

\bibitem{vanErven2021MetaGrad}
T.~V. Erven, W.~M. Koolen, and D.~Van~der Hoeven.
\newblock Metagrad: Adaptation using multiple learning rates in online
  learning.
\newblock {\em Journal of Machine Learning Research}, 22(161):1--61, 2021.

\bibitem{forster2002relative}
J.~Forster and M.~K. Warmuth.
\newblock Relative expected instantaneous loss bounds.
\newblock {\em Journal of Computer and System Sciences}, 64(1):76--102, 2002.

\bibitem{foster1991prediction}
D.~P. Foster.
\newblock Prediction in the worst case.
\newblock {\em The Annals of Statistics}, 19:1084--1090, 1991.

\bibitem{george1986minimax}
E.~I. George.
\newblock Minimax multiple shrinkage estimation.
\newblock {\em The Annals of Statistics}, 14(1):188--205, 1986.

\bibitem{grunwald2021pac}
P.~Gr\"{u}nwald, T.~Steinke, and L.~Zakynthinou.
\newblock {PAC-B}ayes, {MAC-B}ayes and conditional mutual information: Fast
  rate bounds that handle general {VC} classes.
\newblock In {\em Conference on Learning Theory}, pages 2217--2247, 2021.

\bibitem{gyorfi2002nonparametric}
L.~Gy{\"o}rfi, M.~Kohler, A.~Krzyzak, and H.~Walk.
\newblock {\em A distribution-free theory of nonparametric regression}.
\newblock Springer, 2002.

\bibitem{haghifam2021towards}
M.~Haghifam, G.~K. Dziugaite, S.~Moran, and D.~Roy.
\newblock Towards a unified information-theoretic framework for generalization.
\newblock {\em Advances in Neural Information Processing Systems},
  34:26370--26381, 2021.

\bibitem{hannan1957approximation}
J.~Hannan.
\newblock Approximation to {Bayes} risk in repeated play.
\newblock {\em Contributions to the Theory of Games}, 3:97--139, 1957.

\bibitem{hanneke2016refined}
S.~Hanneke.
\newblock Refined error bounds for several learning algorithms.
\newblock {\em Journal of Machine Learning Research}, 17(1):4667--4721, 2016.

\bibitem{hanneke2015minimax}
S.~Hanneke and L.~Yang.
\newblock Minimax analysis of active learning.
\newblock {\em Journal of Machine Learning Research}, 16(1):3487--3602, 2015.

\bibitem{haussler1998sequential}
D.~Haussler, J.~Kivinen, and M.~K. Warmuth.
\newblock Sequential prediction of individual sequences under general loss
  functions.
\newblock {\em IEEE Transactions on Information Theory}, 44(5):1906--1925,
  1998.

\bibitem{juditsky2008mirror}
A.~Juditsky, P.~Rigollet, and A.~B. Tsybakov.
\newblock Learning by mirror averaging.
\newblock {\em The Annals of Statistics}, 36(5):2183--2206, 2008.

\bibitem{kanade2022exponential}
V.~Kanade, P.~Rebeschini, and T.~Vaskevicius.
\newblock Exponential tail local {R}ademacher complexity risk bounds without
  the {B}ernstein condition.
\newblock {\em arXiv preprint arXiv:2202.11461}, 2022.

\bibitem{koltchinskii2011oracle}
V.~Koltchinskii.
\newblock {\em Oracle Inequalities in Empirical Risk Minimization and Sparse
  Recovery Problems}, volume 2033 of {\em {\'E}cole d'{\'E}t{\'e} de
  Probabilit{\'e}s de Saint-Flour}.
\newblock Springer-Verlag Berlin Heidelberg, 2011.

\bibitem{koolen2015second}
W.~M. Koolen and T.~Van~Erven.
\newblock Second-order quantile methods for experts and combinatorial games.
\newblock In {\em Conference on Learning Theory}, pages 1155--1175, 2015.

\bibitem{lecue2009aggregation}
G.~Lecu{\'e} and S.~Mendelson.
\newblock Aggregation via empirical risk minimization.
\newblock {\em Probability Theory and Related Fields}, 145(3):591, 2009.

\bibitem{lecue2013optimality}
G.~Lecu{\'e} and S.~Mendelson.
\newblock On the optimality of the aggregate with exponential weights for low
  temperatures.
\newblock {\em Bernoulli}, 19(2):646--675, 2013.

\bibitem{lecue2014q-aggregation}
G.~Lecu{\'e} and P.~Rigollet.
\newblock Optimal learning with {Q}-aggregation.
\newblock {\em The Annals of Statistics}, 42(1):211--224, 2014.

\bibitem{leung2006information}
G.~Leung and A.~R. Barron.
\newblock Information theory and mixing least-squares regressions.
\newblock {\em IEEE Transactions on Information Theory}, 52(8):3396--3410,
  2006.

\bibitem{liang2015offset}
T.~Liang, A.~Rakhlin, and K.~Sridharan.
\newblock Learning with square loss: localization through offset {Rademacher}
  complexity.
\newblock In {\em Proceedings of The 28th Conference on Learning Theory}, pages
  1260--1285, 2015.

\bibitem{littlestone1989online2batch}
N.~Littlestone.
\newblock From on-line to batch learning.
\newblock In {\em Proceedings of the 2nd annual workshop on Computational
  Learning Theory}, pages 269--284. Morgan Kaufmann Publishers Inc., 1989.

\bibitem{lugosi2023online}
G.~Lugosi and G.~Neu.
\newblock Online-to-{PAC} conversions: {G}eneralization bounds via regret
  analysis.
\newblock {\em arXiv preprint arXiv:2305.19674}, 2023.

\bibitem{massart2007concentration}
P.~Massart.
\newblock {\em Concentration inequalities and model selection. Ecole d'Eté de
  Probabilités de Saint-Flour XXXIII}, volume 1896 of {\em Lecture Notes in
  Mathematics}.
\newblock Springer, 2007.

\bibitem{mcallester1999some}
D.~A. McAllester.
\newblock Some {PAC-Bayesian} theorems.
\newblock {\em Machine Learning}, 37(3):355--363, 1999.

\bibitem{mcallester2003pac}
D.~A. McAllester.
\newblock {PAC-Bayesian} stochastic model selection.
\newblock {\em Machine Learning}, 51(1):5--21, 2003.

\bibitem{mehta2017expconcave}
N.~Mehta.
\newblock Fast rates with high probability in exp-concave statistical learning.
\newblock In {\em Proceedings of the 20th International Conference on
  {A}rtificial {I}ntelligence and {S}tatistics}, pages 1085--1093, 2017.

\bibitem{mendelson2017aggregation}
S.~Mendelson.
\newblock On aggregation for heavy-tailed classes.
\newblock {\em Probability Theory and Related Fields}, 168:641--674, 2017.

\bibitem{mendelson2019unrestricted}
S.~Mendelson.
\newblock An unrestricted learning procedure.
\newblock {\em Journal of the {ACM}}, 66(6):1--42, 2019.

\bibitem{mourtada2022logistic}
J.~Mourtada and S.~Ga{\"\i}ffas.
\newblock An improper estimator with optimal excess risk in misspecified
  density estimation and logistic regression.
\newblock {\em Journal of Machine Learning Research}, 23(31):1--49, 2022.

\bibitem{mourtada2021robust}
J.~Mourtada, T.~Va{\v{s}}kevi{\v{c}}ius, and N.~Zhivotovskiy.
\newblock Distribution-free robust linear regression.
\newblock {\em Mathematical Statistics and Learning}, 4(3):253--292, 2021.

\bibitem{nemirovski2000nonparametric}
A.~Nemirovski.
\newblock Topics in non-parametric statistics.
\newblock {\em Lectures on Probability Theory and Statistics: Ecole d'Ete de
  Probabilites de Saint-Flour XXVIII-1998}, 28:85--277, 2000.

\bibitem{raginsky202110}
M.~Raginsky, A.~Rakhlin, and A.~Xu.
\newblock Information-theoretic stability and generalization.
\newblock {\em Information-Theoretic Methods in Data Science}, 10:302–329,
  2021.

\bibitem{rakhlin2017empirical}
A.~Rakhlin, K.~Sridharan, and A.~B. Tsybakov.
\newblock Empirical entropy, minimax regret and minimax risk.
\newblock {\em Bernoulli}, 23(2):789--824, 2017.

\bibitem{rigollet2012kullback}
P.~Rigollet.
\newblock {Kullback-Leibler} aggregation and misspecified generalized linear
  models.
\newblock {\em The Annals of Statistics}, 40(2):639--665, 2012.

\bibitem{shamir2015sample}
O.~Shamir.
\newblock The sample complexity of learning linear predictors with the squared
  loss.
\newblock {\em Journal of Machine Learning Research}, 16(108):3475--3486, 2015.

\bibitem{steinke2020reasoning}
T.~Steinke and L.~Zakynthinou.
\newblock Reasoning about generalization via conditional mutual information.
\newblock In {\em Conference on Learning Theory}, pages 3437--3452. PMLR, 2020.

\bibitem{steinwart2008svm}
I.~Steinwart and A.~Christmann.
\newblock {\em Support Vector Machines}.
\newblock Information Science and Statistics. Springer-Verlag New York, 2008.

\bibitem{tsybakov2003optimal}
A.~B. Tsybakov.
\newblock Optimal rates of aggregation.
\newblock In {\em Proceedings of the 16th Conference on Computational Learning
  Theory}, pages 303--313, 2003.

\bibitem{vanderhoeven2022}
D.~Van~der Hoeven, N.~Zhivotovskiy, and N.~Cesa-Bianchi.
\newblock A regret-variance trade-off in online learning.
\newblock {\em Advances in Neural Information Processing Systems}, 35, 2022.

\bibitem{vapnik1999nature}
V.~Vapnik.
\newblock {\em The Nature of Statistical Learning Theory}.
\newblock Springer Science \& Business Media, 1999.

\bibitem{vapnik1974theory}
V.~N. Vapnik and A.~Y. Chervonenkis.
\newblock {\em Theory of Pattern Recognition: Statistical Learning Problems (in
  Russian)}.
\newblock Nauka, Moscow, 1974.

\bibitem{vijaykumar2021localization}
S.~Vijaykumar.
\newblock Localization, convexity, and star aggregation.
\newblock {\em Advances in Neural Information Processing Systems},
  34:4570--4581, 2021.

\bibitem{vovk1998mixability}
V.~Vovk.
\newblock A game of prediction with expert advice.
\newblock {\em Journal of Computer and System Sciences}, 56(2):153--173, 1998.

\bibitem{vovk2001competitive}
V.~Vovk.
\newblock Competitive on-line statistics.
\newblock {\em International Statistical Review}, 69(2):213--248, 2001.

\bibitem{wainwright2019high}
M.~J. Wainwright.
\newblock {\em High-dimensional statistics: {A} non-asymptotic viewpoint}.
\newblock Cambridge Series in Statistical and Probabilistic Mathematics.
  Cambridge University Press, 2019.

\bibitem{wintenberger2017boa}
O.~Wintenberger.
\newblock Optimal learning with {Bernstein} online aggregation.
\newblock {\em Machine Learning}, 106(1):119--141, 2017.

\bibitem{yang2000mixing}
Y.~Yang.
\newblock Mixing strategies for density estimation.
\newblock {\em The Annals of Statistics}, 28(1):75--87, 2000.

\bibitem{zhang2006entropy}
T.~Zhang.
\newblock From $\varepsilon$-entropy to {KL}-entropy: Analysis of minimum
  information complexity density estimation.
\newblock {\em The Annals of Statistics}, 34(5):2180--2210, 2006.

\bibitem{zhang2006information}
T.~Zhang.
\newblock Information-theoretic upper and lower bounds for statistical
  estimation.
\newblock {\em IEEE Transactions on Information Theory}, 52(4):1307--1321,
  2006.

\bibitem{zhivotovskiy2018localization}
N.~Zhivotovskiy and S.~Hanneke.
\newblock Localization of {VC} classes: Beyond local {Rademacher} complexities.
\newblock {\em Theoretical Computer Science}, 742:27--49, 2018.

\end{thebibliography}
\end{document}